\documentclass[12pt]{article}
\usepackage{amsfonts}
\usepackage{amssymb}
\usepackage{amsmath}
\usepackage{amsthm}
\usepackage{hyperref}
\usepackage{algpseudocode}
\usepackage{algorithm}
\usepackage{color}
\usepackage{enumitem}
\usepackage{graphicx}
\usepackage{cite}
\usepackage{enumitem}
\usepackage{authblk}
\usepackage[T1]{fontenc}    %

\newtheorem{theorem}{Theorem}
\newtheorem{lemma}[theorem]{Lemma}
\newtheorem{proposition}[theorem]{Proposition}

\newtheorem{definition}[theorem]{Definition}

\newcommand{\real}{{\mathbb R}}

\newcommand{\realinf}{(-\infty, \infty]}

\newcommand{\realn}{\real^n}
\newcommand{\realm}{\real^m}
\newcommand{\realmn}{\real^{m\times n}}

\newcommand{\nat}{\mathbb{N}}

\newcommand{\inner}[2]{\left\langle#1,#2\right\rangle}
\newcommand{\parent}[1]{\left(#1\right)}

\DeclareMathOperator{\dom}{dom}
\DeclareMathOperator{\ri}{ri}

\DeclareMathOperator{\closure}{cl}

\DeclareMathOperator*{\argmin}{arg\,min}
\DeclareMathOperator*{\argmax}{arg\,max}

\newcommand{\set}[2]{\left\{#1 \; \left|\;\; #2 \right.\right\}}

\newcommand{\norm}[1]{\left\|#1\right\|}
\newcommand{\smallnorm}[1]{\lVert{#1}\rVert}
\newcommand{\abs}[1]{\left|#1\right|}

\newcommand{\card}[1]{\left|#1\right|}
\newcommand{\brac}[1]{\left\{#1\right\}}

\DeclareMathOperator{\dist}{dist}
\DeclareMathOperator{\prox}{prox}
\DeclareMathOperator{\bigo}{O}

\newcommand{\transpose}{^{\scriptscriptstyle\top}}

\begin{document}

\title{Two Innovations in Inexact Augmented Lagrangian Methods
       for Convex Optimization\thanks{This work was funded in part by the
       U.S.~Office of Naval Research grant N00014-24-1-2403.}}
\author[1]{Jonathan Eckstein}
\author[1]{Chang Yu}
\affil[1]{\small Department of Management Science and Information Systems,
          Rutgers Business School Newark and New Brunswick,
          Rutgers University}
\date{\today}

\maketitle

\vspace{-3ex}

\begin{abstract}
This paper presents two new techniques relating to inexact solution of
subproblems in augmented Lagrangian methods for convex programming.  The first
involves combining a relative error criterion for solution of the subproblems
with over- or under-relaxation of the multiplier update step.  
This analysis enables a new kind of inexact augmented Lagrangian method
that adapts the amount of
multiplier step relaxation to the accuracy of the subproblem,
subject to a viability test employing the discriminant of a certain quadratic
function.

The second innovation involves solution of augmented Lagrangian subproblems
for problems posed in standard Fenchel-Rockafellar form.  We show that
applying alternating minimization to this subproblem, as in the first two
steps of the ADMM, is equivalent to executing the classical proximal gradient
method on a dual formulation of the subproblem.  By substituting more
sophisticated variants of the proximal gradient method for the classical one,
it is possible to construct new ADMM-like methods with better empirical
performance than using ordinary alternating minimization within an inexact
augmented Lagrangian framework.

The paper concludes by describing some computational experiments exploring
using these two innovations, both separately and jointly, to solve LASSO
problems.
\end{abstract}

\section{Introduction}
\label{sec:introduction}
In augmented Lagrangian methods for convex optimization, it is common to solve
subproblems inexactly, progressively tightening tolerances as the Lagrange
multiplier estimates become more accurate.  This paper presents two new
techniques in this topic area, one relating to the ``outer loop'' --- how the
subproblem tolerance is formulated and the degree of over- or under-relaxation
in the multiplier update step --- and one relating to the ``inner loop'' ---
how the subproblems are solved.  Our computational tests combine these
two innovations.

The first innovation involves combining relative-error solution of subproblems
with over- or under-relaxation of the multiplier update step.  Specifically,
we consider including relaxation of the multiplier estimate update step in the
relative-error augmented Lagrangian method of~\cite{eckstein2013practical}.
The convergence analyses of augmented Lagrangian methods typically allow over-
or under-relaxation of the multiplier update at iteration $k$ by some factor
$\rho_k$, where $\inf_{k} \brac{\rho_k} > 0$ and $\sup_k \brac{\rho_k} < 2$.
The analysis in~\cite{eckstein2013practical}, however, does not have this
feature, although it has other desirable properties:  it uses a relative error
criterion that intrinsically adapts to each problem instance and requires only
a subgradient of the subproblem's effective objective function, not
necessarily an estimate of the distance to the exact subproblem solution or of
the error in the subproblem objective function value. Since over-relaxation
often accelerates proximal algorithms, generalizing the analysis
of~\cite{eckstein2013practical} to include relaxation of the multiplier update
seemed desirable, but proved to be somewhat involved, with the relaxation
factor $\rho_k$ appearing in several places in the subproblem approximation
criterion~\eqref{eq:alt-frame-approx} in Section~\ref{sec:extended-criterion}
below.  There are several ways in which this criterion may be used: on the one
hand, if one treats $\rho_k$ as given,
\eqref{eq:alt-frame-approx} gives an alternative acceptance criterion
for subproblem solution acceptability resembling that
in~\cite{eckstein2013practical}, but with its various terms scaled in
different ways.  Alternatively, however, the acceptance criterion can be used
to determine an acceptable 
relaxation factor
$\rho_k$ from a given inexact subproblem solution.
For example, an approximate subproblem solution that would not be acceptable
to the criterion of~\cite{eckstein2013practical} (which effectively constrains
$\rho_k = 1$) might be acceptable for a smaller value of $\rho_k$.  Below, we
propose a procedure which analyzes any given inexact subproblem solution,
determining the range of $\rho_k$, if any, for which it would be viable.  If
such a range does exist, it is possible to accept the current subproblem
solution and proceed to the multiplier update, using some relaxation factor
within the range. If no acceptable range of $\rho_k$ exists, one must refine
current subproblem solution, after which it may be tested again. To our
knowledge, adapting the length of the multiplier adjustment step to the
quality of the current approximate subproblem solution is a novel development
in augmented Lagrangian methods.  In Section~\ref{sec:extended-criterion}, we
develop the theory of this technique in the general context of parametric
duality --- see \cite[Chapters 29-30]{rockafellar1970convex}
and~\cite{rockafellar1974conjugate} --- making it applicable to essentially
any convex optimization context.  Standard parametric duality 
for convex optimization is briefly
summarized in Section~\ref{sec:parametric-duality}.

The second innovation presented here applies to a more specific 
(yet still very
general) setting, problems in the Fenchel-Rockafellar standard form
$\min_{x\in\realn}\big\{ f(x) + g (Mx) \big\}$ \cite[Chapter
31]{rockafellar1970convex} commonly used to formulate the alternating
direction method of multipliers (ADMM)~\cite{FG83,Gab83,EckBer92}. Applying
the proximal point algorithm to the dual of this problem produces an augmented
Lagrangian method of the form
\begin{align}
(x^{k+1},z^{k+1}) &\in \argmin_{\substack{x\in\real^n \\ z\in\real^m}} 
   \left\{f(x) + g(z) + \inner{p^k}{Mx-z} + \frac{c_k}{2}\norm{Mx-z}^2 \right\} 
   \label{almfenchelmin} \\
p^{k+1} &= p^k + c_k ( Mx^{k+1} - z^{k+1} ) \label{almfenchelupdate},
\end{align}
where $\brac{c_k} \subset \real_{++}$ is a sequence of parameters bounded away
from zero. The $\norm{Mx - z}^2$ term makes the minimand
in~\eqref{almfenchelmin} non-separable 
with respect to $x$ and $z$, leading to a common tactic of cyclically minimizing
it with respect only to the $x$ variables and then only to the $z$ variables.
This approach was tried as early as 1975 in~\cite{GM75}, without much
theoretical justification.  Empirically, this investigation lead to the
variant of the algorithm which updates the multipliers after a single $x$
minimization followed by a single $z$ mininization.  This procedure is the 
ADMM, eventually proved to be convergent in~\cite{FG83,Gab83}.  However, the
analysis of the ADMM commonly requires the proximal parameters $c_k$ to be
constant.

A possible alternative to the ADMM is to use alternating minimization of $f(x)
+ g(z) + \inner{p^k}{Mx-z} + \frac{c_k}{2}\norm{Mx-z}^2$ within some rigorous
  inexact augmented Lagrangian method like the one
  in~\cite{eckstein2013practical}: in this strategy, one iterates alternating
  minimizations over $x$ and $z$ until some optimality tolerance condition is satisfied,
  and then proceeds to the multiplier update $p^{k+1} = p^k + c_k ( Mx^{k+1} -
  z^{k+1} )$, where $(x^{k+1},z^{k+1})$ is the final approximate solution to
  the subproblem.  Here, we call this method ``ALM-ADSS'', standing for
  ``augmented Lagrangian method with alternating direction subproblem solution''
  There are some drawbacks to this approach, the largest being that while the
  number of multiplier updates tends to be much smaller than with the ADMM,
  the total computational effort can be far larger~\cite{PJO}.  Furthermore,
  convergence of the ``inner loop'' of alternating minimization on the
  subproblem $\min_{x,z} \big\{ f(x)
+ g(z) + \inner{p^k}{Mx-z} + \frac{c_k}{2}\norm{Mx-z}^2 \big\}$ is not
  straightforward to prove directly, although one may apply a result of
  Tseng~\cite{Tse01}.  However, Tseng's result, aimed at more general
  situations, does not use the full structure of the subproblem and yields
  only subsequential convergence.

The second innovation in this paper revisits algorithms of this form: an inner
loop applying alternating minimization to the augmented Lagrangian subproblem
in~\eqref{almfenchelmin} inside an inexact augmented Lagrangian algorithm.
Specifically, Section~\ref{sec:connection-between-admm} shows that when $g$ is
``prox friendly'', that is, the proximal map $\prox_{cg}(\bar z) \doteq
\argmin_{z\in\realm}\big\{g(z) + (1/2c)\smallnorm{z-\bar z}^2\big\}$ is
readily computed for any $\bar z \in \realm$ and scalar $c > 0$, a dual
formulation of the subproblem has exactly the form typically suitable for
proximal gradient (forward-backward) methods: minimizing the sum of two
functions: one prox friendly and the other with a Lipschitz continuous
gradient.  Furthermore, applying the proximal gradient method to this dual
subproblem turns out to be identical to using alternating minimization on the
primal formulation of the subproblem.  Since strong convergence results are
known for applying the proximal gradient method to convex Lipschitz-gradient
functions --- see for example~\cite[Sections 10.4 and 10.6]{BeckBook}
--- this insight provides stronger convergence
guarantees than can be obtained from~\cite{Tse01} for the ``inner loop''
procedure of using alternating minimization to (approximately) solve the
augmented Lagrangian subproblem.

However, the insight about the relationship between alternating direction and
forward-backward methods for solving the subproblems
in~\eqref{eq:dual-problem} does more than provide a stronger convergence
result for ALM-ADSS, a method which~\cite{PJO} showed to have disappointing
empirical performance.  Whenever $g$ is prox friendly, the dual formulation of
the subproblem has the form naturally associated with proximal gradient
methods, and the standard proximal gradient method is not the only option for
solving it: as discussed in Section~\ref{sec:admm-like-framework}, one could
instead use any of its more ``modern'', theoretically faster converging
variants.  In particular, we consider using the method of Chambolle and Dossal
in~\cite{chambolle2015convergence}, a modification of the FISTA algorithm of
Beck and Teboulle~\cite{beck2009fast}. Both the Chambolle-Dossal (CD) and
original FISTA methods have theoretically accelerated $\bigo(1/k^2)$ function
value convergence; we employ the CD variant because its theory also asserts
iterate convergence, while the original FISTA analysis only asserts function
value convergence.

We call the general technique of exploiting the proximal-gradient-amenable
structure of the dual of the augmented Lagrangian subproblem ``DEFBAL'',
standing for ``Dual Embedded Forward-Backward Augmented Lagrangian'', and it
could in future be used in conjunction with any variant of the proximal
gradient algorithm.  Each choice of possible proximal-gradient-class method
combined with an inexact augmented Lagrangian ``outer loop'' like the one
in~\cite{eckstein2013practical}, will yield a different ADMM-like algorithm
involving an inner loop containing a minimization of a quadratic perturbation
of $f$ followed by a minimization of a quadratic perturbation of $g$.  This
inner loop continues (unlike in the ADMM) until a subproblem accuracy
criterion is satisfied, after which the multiplier estimates are updated.
Depending on the specific proximal gradient and inexact augmented Lagrangian
variants selected, however, numerous details of the resulting algorithm could
vary.

Section~\ref{sec:numerical-experiments} compares the ordinary ADMM and the
ALM-ADSS method, using the relative error augmented Lagrangian outer loop
of~\cite{eckstein2013practical}, with two  DEFBAL combinations.  The first
combination melds the outer loop of~\cite{eckstein2013practical} with a CD
inner loop.  It has much better performance than ALM-ADSS, empirically
verifying the usefulness of the DEFBAL concept.  However, it does not perform
as well as the ordinary ADMM.  The second combination, however, replaces the
outer loop of~\cite{eckstein2013practical} with the method of
Section~\ref{sec:extended-criterion}, adaptively setting the relaxation
factors.  This resulting combination is performs better than the first
combination, thus verifying the potential usefulness of the method of
Section~\ref{sec:extended-criterion}, an  appears to somewhat outperform the
ADMM, albeit at the cost of some algorithmic complexity.

\section{Parametric Duality and Augmented Lagrangians}
\label{sec:parametric-duality}
To prepare for the analysis of Section~\ref{sec:extended-criterion},
we now briefly review the general parametric duality framework
from~\cite[Chaps. 29-30]{rockafellar1970convex}
and~\cite{rockafellar1974conjugate}. Suppose we have a closed proper
convex function $F:\realn \times \realm \rightarrow \realinf$, and
we wish to solve the primal problem
\begin{equation} \label{eq:parametric-primal}
    \min_{x\in\realn} F(x, 0).
\end{equation}
The second argument of $F$ represents some perturbation of the primal
problem~\eqref{eq:parametric-primal}. Furthermore, $\partial F:
\realn\times\realm\rightrightarrows \realn\times\realm$ denotes the
subgradient mapping of $F$. Define $L$, known as the Lagrangian
of~\eqref{eq:parametric-primal}, to be the concave conjugate of $F$ with respect
to the second argument, that is
\begin{equation}
    L(x,p) = \inf_{u\in \realm} \big\{ F(x,u) - \inner{u}{p} \big\}.
    \label{eq:def-lagrange}
\end{equation}
Next, we define $D$ to be the concave conjugate of $L$ with respect to the
first argument, which is the full concave conjugate of $F$, that is
\begin{equation}
    D(s,p) = \inf_{\substack{x\in \realn \\ u \in \realm}} 
             \big\{ F(x,u) - \inner{u}{p} - \inner{x}{s} \big\},\label{eq:def-dual}
\end{equation}
and the dual problem of~\eqref{eq:parametric-primal} is defined as
\begin{equation}
    \max_{p\in \realm} D(0, p).\label{eq:parametric-dual}
\end{equation}
Furthermore, $\partial D: \realn\times\realm \rightrightarrows 
\realn\times\realm$ denotes the concave subgradient mapping of $D$, that is,
\begin{equation}\label{eq:subgradient-dual}
    (x, u) \in \partial D(s, p) 
    \;\; \Leftrightarrow\;\;
    (\forall\, s'\in\realn, \forall\,p'\in\realm)
    \;\,
    D(s', p') \leq D(s, p) - \inner{x}{s'-s} - \inner{u}{p'-p}.
\end{equation}
A simple application of Fenchel's inequality~\cite[Theorem
23.5]{rockafellar1970convex} shows that weak duality holds, that is, $F(x, 0)
\geq D(0, p)$ for all $x\in\realn$ and $p\in\realm$.
Strong duality, that is, having $x^*\in\real^n$ and $p^*\in\real^m$ such that
$F(x^*,0)=D(0,p^*)$, is equivalent to $(x^*,p^*)$ being a saddle point of $L$,
or $(0,0) \in \partial L(x^*,p^*)$, where $\partial L$ denotes the
convex-concave subgradient map of $L$. The point-to-set maps $\partial F$,
$\partial D$ and $\partial L$ are all maximal monotone, and the following equivalences
hold:
\begin{align}\label{eq:subgradient-relations}
\left(
\begin{array}{l}
\forall\, x,s\in\realn \\
\forall\,u,p\in\realm
\end{array}
\right) \quad
&&(s, p) &\in \partial F(x, u) 
&&\Leftrightarrow & (s, u) &\in \partial L(x, p) 
&&\Leftrightarrow & (x, u) &\in \partial D(s, p).
\end{align}
Generically, an augmented Lagranian algorithm for~\eqref{eq:parametric-primal}
is obtained by applying the proximal point algorithm~\cite{Roc76a} to the dual
function $d(p) \doteq D(0,p)$, as originally shown in~\cite{Roc76b}.  
This procedure takes the form
\begin{align} \label{dualppa}
p^{k+1} = \argmin_{p\in\real^m} \left\{
             d(p) + \frac{1}{2c_k}\norm{p - p^k}
          \right\},
\end{align}
Letting $u^{k+1} \doteq (1/c_k)(p^k - p^{k+1})$, the standard optimality
condition for~\eqref{dualppa} yields that $u^{k+1} \in \partial d(p^{k+1})$
for all $k \geq 0$.  Now, $d(p) = D(Ap)$, where $A$ is the $(n+m) \times m$
matrix taking $p \mapsto (0,p)$, that is, $A = [0 \;\; I]\transpose$.  By the
linear chain rule for subgradients~\cite[Theorem~23.9]{rockafellar1970convex}, $\partial d(p) \supseteq A\transpose \partial
D(Ap)$ (with equality under a standard regularity condition), hence
\begin{align*}
\partial d(p) &\supseteq
    \set{u}{\exists\,x\in\real^n : (x,u) \in \partial D(0,p)} 
 = \set{u}{\exists\,x\in\real^n : (0,p) \in \partial F(x,u)},
\end{align*}
Therefore, $p^{k+1} \in \real^m$ being the result of~\eqref{dualppa} is implied by
\begin{align*}
(\exists\,x^{k+1}\in\realn, u^{k+1}\in\realm) : &&
(0,p^{k+1}) &\in \partial F(x^{k+1},u^{k+1}) &
p^{k+1} = p^k - c_k u^{k+1},
\end{align*}
where the second relation is obtained by solving the equation defining
$u^{k+1}$ for $p^{k+1}$. Substituting this equation into the inclusion on the
left yields
\begin{align}
&& (0,p^k - c_k u^{k+1}) &\in \partial F(x^{k+1},u^{k+1}) \nonumber \\
\Leftrightarrow &&
(0,0) &\in \partial F(x^{k+1},u^{k+1}) + (0,c_k u^{k+1} - p^k)  
   \label{genericAugLagSubGrad} \\
\Leftrightarrow &&
(x^{k+1},u^{k+1})
&\in \argmin_{x\in\realn,u\in\realm}
    \left\{ F(x,u) - \inner{p^k}{u} + \frac{c_k}{2}\norm{u}^2\right\}.
    \label{generalAugLagMin}
\end{align}
Therefore, one way to implement~\eqref{dualppa} is to attempt to 
solve~\eqref{generalAugLagMin}, and then set $p^{k+1} = p^k - c_k u^{k+1}$ if
such a solution is found.\footnote{In the interest of brevity, we omit here a
discussion of the (very mild) conditions under which~\eqref{generalAugLagMin} is
guaranteed to be solvable.}  Our main interest in the next section is the
common situation, applicable to all the commonly used parametric duality
setups, that minimizing $F(x,u) - \inner{p^k}{u} + \frac{c_k}{2}\norm{u}^2$
over $u$ is a straightforward closed-form operation for any given $x$, but the
entire exact mininization~\eqref{generalAugLagMin} may be time consuming and
could thus be desirable to perform approximately.

\section{Relative-Error Over- or Under-Relaxed Augmented
         Lagrangian Methods}
\label{sec:extended-criterion}
This section extends the approximate augmented Lagrangian method of
\cite{eckstein2013practical} to include a relaxation parameter $\rho_k$ in the
multiplier update step --- that is, for some positive scalar $\rho_k$, each
iteration's change to the multiplier estimate $p^k$ is multiplied by $\rho_k$
as compared to the algorithm in~\cite{eckstein2013practical}.  The analysis
resembles that in~\cite{eckstein2013practical}, but contains significant
additional elements.

\subsection{An abstract method of multipliers with relaxation}
\label{sec:abstract-method}
We begin by generalizing (24)-(26) of~\cite{eckstein2013practical} to include
relaxation factors $\rho_k$.  Specifically, suppose we have sequences
$\{x^k\}^\infty_{k=1} \subset \realn$, $\{w^k\}^\infty_{k=0}\subset
\real^{n}$, $\{s^k\}^\infty_{k=1}\subset \real^{n}$,
$\{p^k\}^\infty_{k=0}\subset \realm$, $\{u^k\}^\infty_{k=1}\subset
\realm$, $\{\rho_k\}^\infty_{k=0} \subset \real$,
$\{c_k\}^\infty_{k=0}\subset \real$ with $\inf_k\{c_k\} > 0$, and
$\epsilon \in (0,1)$ that obey the following recursions for all $k \geq 0$:
\begin{align}
    (s^{k+1}, 0)
    & \in \partial F(x^{k+1}, u^{k+1}) + (0,c_k u^{k+1} - p^k) 
    \label{eq:alt-frame-min} \\
    \rho_k\left(\frac{2}{c_k}\right)
    \abs{\inner{x^{k+1} - w^k}{s^{k+1}}}  + \rho_k^2 \norm{s^{k+1}}^2
    & \leq \parent{2\rho_k - \rho_k^2 - \epsilon} \norm{u^{k+1}}^2 
    \label{eq:alt-frame-approx}                        \\
    2\rho_k - \rho_k^2 - \epsilon &\geq 0 \label{eq:alt-frame-eps} \\
    w^{k+1} & = w^k - \rho_k c_k s^{k+1} 
    \label{eq:alt-frame-w}                               \\
    \bar{p}^k & = p^k - c_k u^{k+1} \label{eq:alt-frame-bar-p} \\
    p^{k+1} & = (1-\rho_k) p^k + \rho_k \bar p^k. \label{eq:alt-frame-p}
\end{align}
The initial iterates $p^0\in$ and $w^0\in\realn$ are arbitrary. The
inclusion~\eqref{eq:alt-frame-min} above is equivalent to
\begin{align*}
    (s^{k+1},0) & \in \partial_{x,u}
    \!\left[F(x,u) - \inner{p^k}{u} + \frac{c_k}{2}\norm{u}^2\right]_{\substack{x = x^{k+1} \\ u = u^{k+1}}},
\end{align*}
that is, one minimizes the modified parametric primal function
\begin{equation*}
    F_k(x,u) \doteq F(x,u) - \inner{p^k}{u} + \frac{c_k}{2}\norm{u}^2
    \end{equation*}
approximately with respect to $x$ and exactly with respect to $u$, leaving a
subgradient whose $u$ component is zero and whose $x$ component is $s^{k+1}$.
One may think of $F_k$ as an augmented Lagrangian for the problem of
minimizing $F(x,u)$ subject to $u=0$, with $-p^k$ representing the Lagrange
multiplier estimate and $(c_k/2)\smallnorm{u}^2$ being a quadratic penalty for
violating the constraints. Together with~\eqref{eq:alt-frame-bar-p}, the
condition~\eqref{eq:alt-frame-min} may be seen as an abstract representation
of approximately minimizing the augmented Lagrangian with the penalty
parameter set to $c_k$, after which the ordinary augmented Lagrangian step
would yield the new multiplier vector $\bar p^k$.  The
calculation~\eqref{eq:alt-frame-p} is an over- or under-relaxation of that
standard multiplier adjustment step by a factor of $\rho_k$.
Substituting~\eqref{eq:alt-frame-bar-p} into~\eqref{eq:alt-frame-p} and
simplifying, one may rewrite~\eqref{eq:alt-frame-p} as $p^{k+1}=p^k - \rho_k
c_k u^{k+1}$.

Once~\eqref{eq:alt-frame-min}-\eqref{eq:alt-frame-eps} are satisfied, the
remaining recursions~\eqref{eq:alt-frame-w}-\eqref{eq:alt-frame-p} encode
straightforward, closed-form calculations.  However, starting from any $p^k$,
$w^k$, and $c_k$, finding $s^{k+1}$, $u^{k+1}$, and $\rho_k$ that
satisfy~\eqref{eq:alt-frame-min}-\eqref{eq:alt-frame-eps} will in general be a
nontrivial task requiring some auxiliary iterative ``inner loop'' algorithm.
We do not consider here how this inner loop might be implemented, a question
that we will consider later in Definition~\ref{def:subSolver} and (for a
particular class of parametric primal functions~$F$) and
Section~\ref{sec:admm-like-framework}. For the moment, we simply consider the
properties of sequences satisfying the above recursions.

For any $k \geq 0$, applying the quadratic formula
to~\eqref{eq:alt-frame-eps} yields that an equivalent condition is that $\rho_k$ 
lies between the two values
\begin{equation*}
\frac{-2 \pm \sqrt{2^2 - 4(-1)(-\epsilon)}}{2\cdot(-1)}
= \frac{2 \mp \sqrt{4-4\epsilon}}{2}
= 1 \mp \sqrt{1-\epsilon},
\end{equation*}
that is,
\begin{equation} \label{eq:rho-range}
(\forall\,k\geq 0) \qquad
1 - \sqrt{1-\epsilon} \leq \rho_k \leq 1 + \sqrt{1-\epsilon}.
\end{equation}
Thus, \eqref{eq:alt-frame-eps} implies that $0 < \rho_k < 2$ for all $k$, and
in fact that $\inf_k \rho_k > 0$ and $\sup_k \rho_k < 2$.

Substituting~\eqref{eq:alt-frame-bar-p} into the rightmost term
in~\eqref{eq:alt-frame-min}, one obtains the equivalent inclusions
\begin{align}
&&  
  (s^{k+1}, 0) & \in \partial F(x^{k+1}, u^{k+1}) - (0,\bar p^k) && \nonumber \\
\Leftrightarrow &&
  (s^{k+1}, \bar p^k) & \in \partial F(x^{k+1}, u^{k+1}) \nonumber  \\
\Leftrightarrow &&
  (s^{k+1}, u^{k+1}) & \in \partial L(x^{k+1}, \bar p^{k})
  && [\text{using~\eqref{eq:subgradient-relations}}].  \label{eq:saddle-point-l}
\end{align}
The collection of recursions~\eqref{eq:alt-frame-min}-\eqref{eq:alt-frame-p}
thus remains unaltered if one substitutes~\eqref{eq:saddle-point-l}
for~\eqref{eq:alt-frame-min}.  We now establish an important contractive
property of~\eqref{eq:alt-frame-min}-\eqref{eq:alt-frame-p}:
\begin{lemma} \label{lemma:shrinkage}
    Suppose that $(x^*,p^*)$ is a saddle point of the Lagrangian $L$ as
    defined in~\eqref{eq:def-lagrange}. In the
    setting~\eqref{eq:alt-frame-min}-\eqref{eq:alt-frame-p}, for all $k\geq0$,
    \begin{equation}\label{eq:shrinkage}
        \norm{p^{k+1}-p^*}+\norm{w^{k+1} - x^*}^2 \leq \norm{p^k - p^*}^2 + \norm{w^k - x^*}^2 - \epsilon c_k^2 \norm{u^{k+1}}^2,
    \end{equation}
    and therefore the sequence $\big\{ \smallnorm{p^{k+1} - p^{*}}^2 +
    \smallnorm{w^{k+1}-x^*}^2 \big\}$ is nonincreasing.
\end{lemma}
\begin{proof}
Fixing any $k \geq 0$,
\begin{align} 
        \norm{p^{k+1} - p^*}^2 & = \norm{p^k - \rho_k c_k u^{k+1} - p^*}^2                       \nonumber \\
                               & = \norm{p^k - p^*}^2 - 2\rho_k c_k\inner{u^{k+1}}{p^k - p^*} + \rho_k^2 c_k^2\norm{u^{k+1}}^2
                               \nonumber \\
                               & = \norm{p^k - p^*}^2 - 2\rho_k c_k\inner{u^{k+1}}{p^k - \bar{p}^k + \bar{p}^k - p^*} + \rho_k^2 c_k^2\norm{u^{k+1}}^2
                               \nonumber \\
                               & = \norm{p^k - p^*}^2 - 2\rho_k c_k\inner{u^{k+1}}{p^k - \bar{p}^k} - 2\rho_k c_k\inner{u^{k+1}}{\bar{p}^k - p^*} + \rho_k^2 c_k^2\norm{u^{k+1}}^2 
                               \nonumber \\
                               & \overset{\text{(*)}}{=} \norm{p^k - p^*}^2 - 2\rho_k c_k^2\norm{u^{k+1}}^2 - 2\rho_k c_k\inner{u^{k+1}}{\bar{p}^k - p^*} + \rho_k^2 c_k^2\norm{u^{k+1}}^2
                               \nonumber \\
                               & = \norm{p^k - p^*}^2 + (\rho_k^2 - 2 \rho_k) c_k^2\norm{u^{k+1}}^2 - 2\rho_k c_k\inner{u^{k+1}}{\bar{p}^k - p^*},
                               \label{eq:shrink-p-adaptive}
\end{align}
where the equality (*) uses that $p^k - \bar{p}^k = c_k u^{k+1}$.
Next, consider the quantity $\norm{w^{k+1}-x^*}^2$, which may be rewritten
\begin{align}
    \norm{w^{k+1} - x^*}^2
    & = \norm{w^k - \rho_k c_k s^{k+1} - x^*}^2  \nonumber         \\
    & = \norm{w^k - x^*}^2 - 2\rho_k c_k\inner{w^k - x^*}{s^{k+1}} + \rho_k^2 c_k^2\norm{s^{k+1}}^2  \nonumber     \\
    & = \norm{w^k - x^*}^2 - 2\rho_k c_k\inner{w^k - x^{k+1} + x^{k+1} - x^*}{s^{k+1}} + \rho_k^2 c_k^2\norm{s^{k+1}}^2  \nonumber \\
    & = \norm{w^k - x^*}^2 - 2\rho_k c_k\inner{w^k - x^{k+1}}{s^{k+1}}  \nonumber  \\
    & \qquad\qquad 
    - 2\rho_k c_k\inner{x^{k+1} - x^*}{s^{k+1}} + \rho_k^2 c_k^2\norm{s^{k+1}}^2.
    \label{eq:shrink-w-adaptive}
    \end{align}
Adding~\eqref{eq:shrink-w-adaptive} to~\eqref{eq:shrink-p-adaptive} produces
\begin{align}
\norm{p^{k+1}-p^*}+\norm{w^{k+1} - x^*}^2
    & = \norm{p^k - p^*}^2 + \norm{w^k - x^*}^2 + (\rho_k^2-2\rho_k) c_k^2 \norm{u^{k+1}}^2 \nonumber \\
    & \qquad -2\rho_k c_k \inner{u^{k+1}}{\bar{p}^k - p^*}- 2\rho_k c_k \inner{x^{k+1} - x^*}{s^{k+1}} \nonumber \\
    & \qquad -2\rho_k c_k \inner{w^k - x^{k+1}}{s^{k+1}} + \rho_k^2 c_k^2 \norm{s^{k+1}}^2 \nonumber \\
    & = \norm{p^k - p^*}^2 + \norm{w^k - x^*}^2 - (2\rho_k - \rho_k^2) c_k^2 \norm{u^{k+1}}^2   \nonumber  \\
    & \qquad - 2\rho_k c_k \Big[\inner{\bar{p}^k - p^*}{u^{k+1}}  + \inner{x^{k+1} - x^*}{s^{k+1}} \Big]  \nonumber  \\
    & \qquad - 2\rho_k c_k \inner{w^k -x^{k+1}}{s^{k+1}} + \rho_k^2 c_k^2 \norm{s^{k+1}}^2.
    \label{eq:shrink-pw-eq-adaptive}
\end{align}
Since $\big(x^*, p^* \big)$ was assumed to be a saddle point of $L$, we have $(0,
0)\in \partial L\parent{x^*,p^*}$, which, together with the monotonicity of
$\partial L$ and $\big(s^{k+1},u^{k+1}\big) \in \partial L\big( x^{k+1},
\bar{p}^k \big)$ from~\eqref{eq:saddle-point-l}, means that
\begin{equation}
    \inner{x^{k+1} - x^*}{s^{k+1}-0} + \inner{\bar{p}^k - p^*}{u^{k+1}-0} \geq 0.\label{eq:monotonicity}
    \end{equation}
Substituting~\eqref{eq:monotonicity} into~\eqref{eq:shrink-pw-eq-adaptive}
and using that both $\rho_k>0$ and $c_k>0$, one obtains
\begin{multline}
\norm{p^{k+1}-p^*}+\norm{w^{k+1} - x^*}^2
     \leq \norm{p^k - p^*}^2 + \norm{w^k - x^*}^2 - (2\rho_k-\rho_k^2)  c_k^2 \norm{u^{k+1}}^2   \\
     - 2\rho_k c_k \inner{w^k - x^{k+1}}{s^{k+1}} + \rho_k^2 c_k^2\norm{s^{k+1}}^2.
    \label{eq:shrink-pw-almost-adaptive}
\end{multline}
By the design of the error criterion~\eqref{eq:alt-frame-approx},
the last two terms of~\eqref{eq:shrink-pw-almost-adaptive} may be
``buried'' in the term $c_k \norm{u^{k+1}}$. More specifically,
\begin{align}
        -2 \rho_k c_k \inner{w^k - x^{k+1}}{s^{k+1}} + \rho_k^2 c_k^2 \norm{s^{k+1}}^2
         & = c_k^2 \Big[\parent{\frac{- 2\rho_k }{c_k}} \inner{w^k - x^{k+1}}{s^{k+1}}
         + \rho_k^2 \norm{s^{k+1}}^2 \Big]  \nonumber \\
         & \leq c_k^2 \Big[\rho_k \parent{\frac{2}{c_k}} \abs{\inner{w^k - x^{k+1}}{s^{k+1}}} + \rho_k^2 \norm{s^{k+1}}^2 \Big] \nonumber \\
         & \leq (2\rho_k - \rho_k^2 - \epsilon) c_k^2 \norm{u^{k+1}}^2,
         \label{eq:bury-error-terms}
\end{align}
with the second inequality following from~\eqref{eq:alt-frame-approx}.
Substituting~\eqref{eq:bury-error-terms} into~\eqref{eq:shrink-pw-almost-adaptive}
then produces
\begin{equation} 
\label{eq:shrink-pw-adaptive}
    \norm{p^{k+1}-p^*}+\norm{w^{k+1} - x^*}^2 
    \leq \norm{p^k - p^*}^2 + \norm{w^k - x^*}^2 - \epsilon c_k^2 \norm{u^{k+1}}^2,
\end{equation}
which is~\eqref{eq:shrinkage}.  Since $\epsilon > 0$ and~\eqref{eq:shrinkage} holds
for arbitrary $k \geq 0$, it is immediate that the sequence $\big\{
\smallnorm{p^{k+1} - p^{*}}^2 +
\smallnorm{w^{k+1}-x^*}^2 \big\}$ must be nonincreasing.
\end{proof}

The inequality~\eqref{eq:shrinkage} establishes Fej\'er monotonicity of the
sequence $\big\{(p^k,w^k)\big\}$ to the set of saddle points of $L$.  The
resulting convergence properties are summarized in the following proposition:

\begin{proposition}\label{prop:convergence}
    Suppose there exists at least one saddle point of the Lagrangian $L$
    defined in~\eqref{eq:def-lagrange}. In the
    setting~\eqref{eq:alt-frame-min}-\eqref{eq:alt-frame-p}, including the
    immediately preceding assumptions, the following hold:
    \begin{itemize}[nosep,label=-]
        \item The sequences $\{\bar p^k\}$, $\{p^k\}$ and $\{w^k\}$ are bounded.
        \item $u^k \rightarrow 0$, $c_k u^{k+1} \rightarrow 0$ and $s^k \rightarrow 0$.
        \item $\big\{F(x^k, u^k)\big\}_{k=0}^\infty$ converges to the optimal value of 
        the primal problem~\eqref{eq:parametric-primal}, which is the same as
        the optimal value of the dual problem~\eqref{eq:parametric-dual}.
        \item $\big\{D(s^{k+1}, \bar p^k\big\}_{k=0}^\infty$ also converges to
        the optimal value of the primal and dual
        problems~\eqref{eq:parametric-primal} and~\eqref{eq:parametric-dual}.
        \item All accumulation points of $\{x^k\}$ are solutions to the primal problem~\eqref{eq:parametric-primal}.
        \item All accumulation points of $\{\bar p^k\}$ and $\{p^k\}$ are solutions to the dual problem~\eqref{eq:parametric-dual}.
    \end{itemize}
\end{proposition}
\begin{proof}
Let $(x^*,p^*)$ be any saddle point of $L$, assumed to exist by hypothesis.
Since~\eqref{eq:shrinkage} holds for all $k \geq 0$ by Lemma~\ref{lemma:shrinkage}, we may sum it over $k$ to yield
\begin{align}
    \norm{p^{k+1}-p^*}^2 + \norm{w^{k+1} - x^*}^2 
    &\leq \norm{p^0 - p^*}^2 + \norm{w^0 - x^*}^2 
           - \epsilon \sum_{i=0}^k c_i^2 \norm{u^{i+1}}^2 \label{eq:telescope} \\
    &\leq \norm{p^0 - p^*}^2 + \norm{w^0 - x^*}^2         \nonumber
\end{align}
for all $k \geq 0$. Therefore, $\{p^k\}$ and
$\{w^k\}$ must be bounded. Since the left side of~\eqref{eq:telescope} is
nonnegative, the sequence $\brac{c_k^2
\smallnorm{u^{k+1}}^2}$ must be summable, so $c_k u^{k+1} \rightarrow 0$.
Since $c_k$ is bounded away from zero, $\brac{\smallnorm{u^{k+1}}^2}$ is also
summable and so $u^{k+1} \rightarrow 0$.  

Since $p^k - \bar{p}^k = c_k u^{k+1}$ for all $k\geq 0$, having $c_k u^{k+1}
\rightarrow 0$ means that $p^k - \bar p^k \rightarrow 0$, so that 
$\{p^k\}$ and $\{\bar p^k\}$ have identical sets of limit points and the
boundedness of $\{p^k\}$ implies boundedness of $\{\bar p^k\}$. Furthermore,
since $0 < \rho_k < 2$ for all $k$ as shown in~\eqref{eq:rho-range}, it is
also the case that $\rho_k c_k u^{k+1} \to 0$ and hence that $p^k - p^{k+1}
\rightarrow 0$, since $p^k - p^{k+1} = \rho_k c_k u^{k+1}$ 
for all $k \geq 0$ by simple algebraic
manipulation of~\eqref{eq:alt-frame-p}.
Multiplying~\eqref{eq:alt-frame-approx} by $c_k^2$ or restating the final
inequality in~\eqref{eq:bury-error-terms}, one has
\begin{equation}  \label{eq:squeeze-cond}
(\forall\,k \geq 0) \qquad
2 c_k \rho_k \abs{\inner{w^k - x^{k+1}}{s^{k+1}}} + \rho_k^2 \norm{s^{k+1}}^2 
   \leq (2\rho_k - \rho_k^2 - \epsilon) c_k^2 \norm{u^{k+1}}^2.
\end{equation}
Since the maximum possible value of $2\rho_k - \rho_k^2 - \epsilon$ is
$1-\epsilon$, as may be seen by an elementary analysis of the concave quadratic function
$\rho \mapsto 2\rho - \rho^2 - \epsilon$, having $c_k u^{k+1} \to 0$ implies
that the right side of~\eqref{eq:squeeze-cond} converges to zero.  The left side,
being nonnegative, must also converge to zero, and hence its two constituent
nonnegative sequences $\brac{2\rho_k c_k \abs{\inner{w^k -
x^{k+1}}{s^{k+1}}}}$ and $\brac{\rho_k^2 c_k^2 \smallnorm{s^{k+1}}^2}$
converge to zero as well. Using that $\brac{\rho_k}$ and $\brac{c_k}$ are
bounded away from zero (see~\eqref{eq:rho-range} for the case of
$\{\rho_k\}$), it follows that $\brac{\smallnorm{s^{k+1}}^2}$ and
$\brac{\abs{\inner{w^k - x^{k+1}}{s^{k+1}}}}$ also converge to zero.
Immediately, one may deduce that $\inner{w^k - x^{k+1}}{s^{k+1}} \rightarrow 0$ and
$s^{k+1} \rightarrow 0$. Next, we can write
\begin{equation}\label{eq:decompose-inner}
    \inner{x^{k+1}}{s^{k+1}} = \inner{w^k}{s^{k+1}} - \inner{w^k-x^{k+1}}{s^{k+1}},
\end{equation}
and notice that since $\{w^k\}$ is bounded and $s^{k+1} \rightarrow 0$, one
has $\inner{w^k}{s^{k+1}} \rightarrow 0$. Since it has already been
established that $\inner{w^k-x^{k+1}}{s^{k+1}} \rightarrow 0$, it follows that
$\inner{x^{k+1}}{s^{k+1}} \rightarrow 0$. 

We now consider the dual problem.  Applying the last equivalence
in~\eqref{eq:subgradient-relations} to $(s^{k+1}, u^{k+1})\in \partial
L(x^{k+1}, \bar{p}^k)$ as established in~\eqref{eq:saddle-point-l} yields
$(x^{k+1}, u^{k+1})\in \partial D(s^{k+1}, \bar{p}^k)$ for all $k\geq 0$.
Using the subgradient inequality in~\eqref{eq:subgradient-dual} with
$s=s^{k+1}$, $p=\bar p^k$, $s'=0$, $p'=p^*$, $x=x^{k+1}$, and $u=u^{k+1}$
yields for all $k\geq 0$ that
\begin{equation} \label{dualSqueeze1}
    D(0, p^*) \leq D(s^{k+1}, \bar{p}^k) - \inner{x^{k+1}}{0 - s^{k+1}} - \inner{u^{k+1}}{p^* - \bar{p}^k}.
\end{equation}
Isolating the first term on the right of this inequality leads to
\begin{equation*}
    D(s^{k+1}, \bar p^{k}) \geq D(0, p^*) - \inner{x^{k+1}}{s^{k+1}} + \inner{u^{k+1}}{p^* - \bar p^k}.
\end{equation*}
Passing to the limit $k \rightarrow \infty$ and using $u^{k+1} \to 0$,
$\inner{x^{k+1}}{s^{k+1}} \rightarrow 0$ (as recently shown), and the
boundedness of $\{\bar p^k\}$, we have
\begin{equation}\label{eq:liminf-dual}
    \liminf_{k \rightarrow \infty} D(s^{k+1}, \bar p^k) \geq D(0, p^*).
\end{equation}
Let $\mathcal{K}$ be a subsequence such that $D(s^{k+1}, \bar p^k)
\rightarrow_{\mathcal{K}} \limsup_{k \rightarrow \infty} D(s^{k+1}, \bar p^k)$.
By the boundedness of $\{\bar p^k\}$, there exists a convergent subsequence
$\mathcal{K}' \subseteq \mathcal{K}$ such that $\{\bar p^k\}_{k \in
\mathcal{K}'}$ converges to some limit $p^{\infty}$. We then observe that
\begin{align*}
    \qquad\qquad D(0, p^*) 
    & \geq D(0, p^{\infty})    && \text{[since $p^*$ is dual optimal]} \\
    & = D(\lim_{\substack{k \rightarrow \infty \\ k \in \mathcal{K}'}} s^{k+1}, \lim_{\substack{k \rightarrow \infty \\ k \in \mathcal{K}'}} \bar p^k) && \text{[since $s^{k+1} \rightarrow 0$, $\bar p^k \rightarrow_{\mathcal{K}'} p^{\infty}$]} \\
    & \geq \limsup_{\substack{k \rightarrow \infty \\ k \in \mathcal{K}'} } D(s^{k+1}, \bar p^k) && [\text{by the upper semicontinuity of~}D] \\
    & = \limsup_{k \rightarrow \infty} D(s^{k+1}, \bar p^k) && \text{[by the choice of $\mathcal{K}'$]}.
\end{align*}
In conjunction with~\eqref{eq:liminf-dual}, one then has
\begin{equation*}
    \liminf_{k \rightarrow \infty} D(s^{k+1}, \bar p^k) \geq D(0, p^*) \geq \limsup_{k \rightarrow \infty} D(s^{k+1}, \bar p^k),
\end{equation*}
so $\lim_{k \rightarrow \infty} D(s^{k+1}, \bar p^k) = D(0, p^*)$, that is, the
sequence of $D(s^{k+1}, \bar p^k)$ converges to the optimal value of the dual
problem~\eqref{eq:parametric-dual}.  If a saddle point of the Lagrangian
exists, this dual optimal value is the same as the optimal value of the primal problem.

Next, we turn to the primal problem. Since $(x^{k+1}, u^{k+1})\in \partial
D(s^{k+1}, \bar p^k)$, Fenchel's equality implies
\begin{equation*}
    F(x^{k+1}, u^{k+1}) 
    = D(s^{k+1}, \bar p^k) + \inner{x^{k+1}}{s^{k+1}} + \inner{u^{k+1}}{\bar p^k},
\end{equation*}
Passing to the $k \rightarrow
\infty$ limit and using that $\bar p^k$ is bounded, $u^{k+1} \rightarrow 0$, 
$\inner{x^{k+1}}{s^{k+1}} \rightarrow 0$, and $D(s^{k+1}, \bar p^k) \to D(0,
p^*)$ as just demonstrated, we have $F(x^{k+1}, u^{k+1}) \rightarrow D(0, p^*)
= F(x^*, 0)$.

It remains only to show that all accumulation points of $\{x^k\}$ are
solutions to the primal problem~\eqref{eq:parametric-primal} and all
accumulation points of $\{\bar p^k\}$ and $\{p^k\}$ are solutions to the dual
problem~\eqref{eq:parametric-dual}. Consider any  accumulation point
$x^{\infty}$ of $\{x^k\}$ with corresponding index subsequence $\mathcal{K}''
\subseteq \nat$. Since $u^{k+1} \rightarrow 0$ and $F$ is lower semicontinuous
and convex, we have
\begin{equation*}
    F(x^{\infty}, 0) \leq \liminf_{\substack{k \rightarrow \infty \\ k \in \mathcal{K}''}} F(x^{k+1}, u^{k+1}) = F(x^*, 0).
\end{equation*}
However, since $x^*$ is optimal, we must have $F(x^\infty, 0) \geq F(x^*, 0)$,
which leads to $F(x^{\infty}, 0) = F(x^*, 0)$.  The choice of the limit point
$x^\infty$ being arbitrary, one may conclude that all limit points of
$\{x^k\}$ are optimal for the primal problem.

Similarly, consider any limit
point $\bar p^\infty$ of $\{\bar p^k\}$, with corresponding subsequence
$\mathcal{K}'''$. Using that $s^{k+1} \rightarrow 0$ and that $D$ is upper
semicontinuous, 
\begin{equation*}
D(0,\bar p^\infty) \geq 
\limsup_{\substack{k\to\infty \\ k\in\mathcal{K}'''}} D(s^{k+1},\bar p^k) = D(0,p^*),
\end{equation*}
and so $\bar p^\infty$ is also optimal for the dual problem.  Thus, all limit
points of $\{\bar p^k\}$ are dual optimal.  Further, since it has already been
established that $\{\bar p^k\}$ and $\{p^k\}$ share the same set of limit
points, all accumulation points of $\{p^k\}$ are dual optimal as well.
\end{proof}

\begin{proposition}\label{prop:unbounded}
Consider the same setting as in Proposition~\ref{prop:convergence}, except
that no saddle point of the Lagrangian $L$ exists.  Then at least one of the
sequences $\{x^k\}$ or $\{\bar p^k\}$ is unbounded.
\end{proposition}
\begin{proof}
See~\ref{app:unbounded-proof}.
\end{proof}

\subsection{A special case of the duality framework: problems in Fenchel-Rockafellar form}
\label{sec:sum-of-two-convex}
The following special case of the parametric duality framework will receive
particular attention later in this paper: let
$f:\realn\rightarrow\mathbb{R}\cup\{+\infty\}$ and
$g:\realm\rightarrow\mathbb{R}\cup\{+\infty\}$ be closed proper convex
functions, and $M\in\realmn$ and consider the standard Fenchel-Rockafellar
problem
\begin{equation} \label{eq:primal-problem}
    \min_{x\in\realn} f(x)+g(Mx).
\end{equation}
One way to model this problem within the parametric duality formalism of
Section~\ref{sec:parametric-duality} is to replace $n \leftarrow n+m$ and
define $F : \real^{n+m}\times\realm \to \real \cup \{+\infty\}$ by
\begin{align} \label{eq:sum-F}
    F\big( (x,z),u \big) & =
    \begin{cases}
        f(x) + g(z), & \mbox{if } Mx - z + u = 0 \\
        +\infty,     & \mbox{otherwise}.
    \end{cases}
\end{align}
Using that $f$ and $g$ are closed proper convex, it is easily verified that $F$
is closed proper convex. The primal problem~\eqref{eq:parametric-primal} may
then be written
\begin{equation} \label{eq:conjugate-primal}
    \min_{x\in \realn, z\in \realm}F\big((x,z),0\big).
\end{equation}
Specializing~\eqref{eq:def-lagrange}, the Lagrangian function $L$ on $\real^{n+m}
\times \realm$ is obtained by taking the concave conjugate of $F$ with respect
to the second argument, yielding
\begin{align}
    L\big( (x,z),p \big) & = \inf_{u \in \realm} \left\{F\big( (x,z),u \big) - \inner{u}{p}\right\} \nonumber \\
                            & = f(x) + g(z) + \inner{p}{Mx-z}.
                            \label{eq:conjugate-lagrange}
\end{align}
The dual function $D$ is then obtained by taking the concave conjugate of $L$
with respect to its first argument, yielding, for any $s_x\in\real^n$ and
$s_z,p \in \real^m$,
\begin{align}
D\big((s_x,s_z),p\big)
&= \inf_{\substack{x\in\real^n\\z\in\real^m}}
   \Big\{ f(x) + g(z) + \inner{p}{Mx-z} 
          - \big\langle{(s_x,s_z)},{(x,z)}\big\rangle \Big\} \nonumber \\
&= \inf_{x\in\real^n} \big \{ f(x) + \inner{M\transpose p - s_x}{x}\big\}
   + \inf_{z\in\real^m} \big\{ g(z) - \inner{p+s_z}{z} \big\} \nonumber \\
&= -f^*(s_x - M\transpose p) - g^*(p + s_z). \label{fenchelExtendedDual}
\end{align}
In particular, the parametric dual problem of maximizing $D(0,p)$ is thus
equivalent to maximizing $-f^*(-M\transpose p) - g^*(p)$, which is in turn
equivalent to minimizing $f^*(-M\transpose p) + g^*(p)$, which is the standard
Fenchel-Rockafellar dual of~\eqref{eq:primal-problem}.

Returning now to the expression for $L$ in~\eqref{eq:conjugate-lagrange},
observe that at any point $\big( (x,z), p \big) \in \real^{n+m} \times
\realm$, one may readily compute
\begin{equation} \label{eq:conj-lag-sub}
    \partial L\big( (x,z),p \big) = \big( \parent{\partial f(x) + M\transpose p} 
    \times \parent{\partial g(z) - p} \big) \times \big\{z - Mx \big\}.
\end{equation}

We now consider how the
recursions~\eqref{eq:alt-frame-min}-\eqref{eq:alt-frame-p} specialize in the
case of the particular class of parametric primal functions~$F$ described
in~\eqref{eq:sum-F}. Substituting the form of $L$ given
in~\eqref{eq:conj-lag-sub} into~\eqref{eq:saddle-point-l} with the sequence
$\{x^k\} \subset \real^n$ rewritten as $\big\{(x^k,z^k)\big\}
\subset \real^{n+m}$ yields together with~\eqref{eq:alt-frame-bar-p} the
relations
\begin{align*}
u^{k+1} &= z^{k+1} - Mx^{k+1} \\
\bar p^k &= p^k + c_k(Mx^{k+1} - z^{k+1}) \\
s^{k+1} &\in \parent{\partial f(x^{k+1}) + M\transpose \bar p^k} 
               \times \parent{\partial g(z^{k+1}) - \bar p^k}.
\end{align*}
Rewriting $x^{k+1} \leftarrow (x^{k+1},z^{k+1})$ and substituting $u^{k+1} =
z^{k+1} - Mx^{k+1}$ throughout, the recursion
ensemble~\eqref{eq:saddle-point-l},
\eqref{eq:alt-frame-approx}-\eqref{eq:alt-frame-p} may then be expressed for the
particular setup of~$F$ in~\eqref{eq:sum-F} and therefore~$L$
in~\eqref{eq:conjugate-lagrange} as
\begin{align}
\bar p^k &= p^k + c_k(Mx^{k+1} - z^{k+1}) \label{eq:sum-frame-bar-p} \\
s^{k+1} &\in \parent{\partial f(x^{k+1}) + M\transpose \bar p^k} 
               \times \parent{\partial g(z^{k+1}) - \bar p^k} \label{eq:sum-frame-s} \\
    \frac{2\rho_k}{c_k}
    \Big\lvert\big\langle(x^{k+1},z^{k+1}) - w^k, s^{k+1}\big\rangle\Big\rvert  + \rho_k^2 \norm{s^{k+1}}^2
    & \leq \parent{2\rho_k - \rho_k^2 - \epsilon} \norm{Mx^{k+1}-z^{k+1}}^2 
        \label{eq:sum-frame-approx}                        \\
    2\rho_k - \rho_k^2 - \epsilon &\geq 0 \label{eq:sum-frame-eps} \\
    w^{k+1} & = w^k - \rho_k c_k s^{k+1} \label{eq:sum-frame-w}                                 \\
    p^{k+1} & = (1-\rho_k) p^k + \rho_k \bar p^k. \label{eq:sum-frame-p}
\end{align}
In this rewriting, each $s^k$ and $w^k$ is a member of $\real^{n+m}$.  To
distinguish between the two block components of these vectors when necessary,
we will write $s^k = (s_x^k,s_z^k)$ and $w^k = (w_x^k,w_z^k)$, where
$s_x^k,w_x^k \in \real^n$ and $s_z^k,w_z^k \in \real^m$, for all $k \geq 1$.
In this case, \eqref{eq:sum-frame-s} may be expressed as the inclusions
\begin{align} \label{eq:sum-frame-s2}
s_x^{k+1} &\in \partial f(x^{k+1}) + M\transpose \bar p^k &
s_z^{k+1} &\in \partial g(z^{k+1}) - \bar p^k.
\end{align}
Substituting~\eqref{eq:sum-frame-bar-p} into these inclusions yields
\begin{align}
s_x^{k+1} &\in \partial f(x^{k+1}) + M\transpose p^k + c_k M\transpose(Mx^{k+1} - z^{k+1})
    \label{eq:sum-frame-subgrad-x} \\
s_z^{k+1} &\in \partial g(z^{k+1}) - p^k + c_k(z^{k+1} - Mx^{k+1}),
    \label{eq:sum-frame-subgrad-z}
\end{align}
or equivalently
\begin{equation} \label{eq:approx-min-sum-auglag}
(s_x^k,s_z^k) \in
\partial \! \left[
f(x) + g(z) + \inner{p^k}{Mx-z} + \frac{c_k}{2}\norm{Mx - z}^2
\right]_{\substack{x = x^{k+1} \\ z = z^{k+1}}}.
\end{equation}
Therefore, the conditions
\eqref{eq:sum-frame-bar-p}-\eqref{eq:sum-frame-s} or
\eqref{eq:sum-frame-s2} may be viewed as an approximate minimization
of the augmented Lagrangian $f(x) + g(z) + \inner{p^k}{Mx-z} +
({c_k}/{2})\norm{Mx - z}^2$ in~\eqref{almfenchelmin}.  In particular, having
$s_x^k=0$ and $s_z^k=0$ corresponds to exactly minimizing the augmented
Lagrangian.

\subsection{Adaptive relaxation}
\label{sec:ways-to-use}
Returning for the time being to the general case of $F$, we now consider how
to employ the approximation criterion~\eqref{eq:alt-frame-approx}, given fixed
value of $\epsilon> 0$.  One possibility is to treat the under/over-relaxation
parameter $\rho_k$ as given and use~\eqref{eq:alt-frame-approx} as a stopping
criterion for approximately solving the subproblem~\eqref{generalAugLagMin}.
In this approach, one chooses $\rho_k$ at the beginning of iteration $k$ and
uses some iterative method to solve~\eqref{generalAugLagMin}, periodically
checking whether~\eqref{eq:alt-frame-approx} holds.  If so, the subproblem has
been solved to sufficient accuracy for the selected $\rho_k$, and one can
proceed to the updates~\eqref{eq:alt-frame-w}-\eqref{eq:alt-frame-p} and the
subsequent iteration.  If not, one needs to return to subproblem and solve it
more accurately, which should eventually lead to a smaller $s^{k+1}$ and thus
smaller values $\abs{\inner{x^{k+1} - w^k}{s^{k+1}}}$ and
$\smallnorm{s^{k+1}}^2$.  This approach is a direct generalization of that
taken in~\cite{eckstein2013practical}, allowing for varying $\rho_k$ instead
of $\rho_k \equiv 1$.

Another option, however, is to dynamically adjust $\rho_k$.  In this approach,
one still uses an iterative method to solve~\eqref{generalAugLagMin},
periodically checking whether it can be terminated.  But the termination check
is more elaborate, attempting to find a value of $\rho_k$ for
which~\eqref{eq:sum-frame-approx} and~\eqref{eq:alt-frame-eps} are satisfied.
Suppose that the iterative solution procedure for~\eqref{generalAugLagMin} has
been paused with particular candidate values of $x^{k+1}$, $s^{k+1}$, and
$u^{k+1}$ satisfying~\eqref{eq:alt-frame-min}. One may then rearrange the
error criterion~\eqref{eq:alt-frame-approx} into a quadratic inequality in
$\rho_k$, namely
\begin{equation} \label{eq:rhok-ineq}
    \parent{\norm{u^{k+1}}^2 + \norm{s^{k+1}}^2} \rho_k^2 - 2\parent{\norm{u^{k+1}}^2 - \frac{1}{c_k}\abs{\inner{x^{k+1} - w^k}{s^{k+1}}}} \rho_k + \epsilon \norm{u^{k+1}}^2 \leq 0.
\end{equation}
This inequality has a solution if its discriminant is nonnegative, that is, if
\begin{equation}\label{eq:discriminant}
    4 \parent{\norm{u^{k+1}}^2 - \frac{1}{c_k}\abs{\inner{x^{k+1} - w^k}{s^{k+1}}}}^2 - 4 \epsilon \parent{\norm{u^{k+1}}^4 + \norm{u^{k+1}}^2\norm{s^{k+1}}^2} \geq 0.
\end{equation}
Dividing by the common factor of $4$, an equivalent condition is
\begin{equation} \label{eq:defDeltak}
    \Delta_k \doteq \parent{\norm{u^{k+1}}^2 - \frac{1}{c_k}\abs{\inner{x^{k+1} - w^k}{s^{k+1}}}}^2 - \epsilon \parent{\norm{u^{k+1}}^4 + \norm{u^{k+1}}^2\norm{s^{k+1}}^2} \geq 0,
\end{equation}
with the discriminant defined in~\eqref{eq:discriminant} being $4\Delta_k$.
If $\Delta_k \geq 0$, then solutions to~\eqref{eq:rhok-ineq} must exist and, once again applying the quadratic formula, must lie between the two values
\begin{align*}
\frac{2\left(\norm{u^{k+1}}^2 - \frac{1}{c_k}\abs{\inner{x^{k+1} - w^k}{s^{k+1}}}\right)
      \pm \sqrt{4\Delta_k}}{2\left(\norm{u^{k+1}}^2 + \norm{s^{k+1}}^2\right)}
      =
\frac{\norm{u^{k+1}}^2 - \frac{1}{c_k}\abs{\inner{x^{k+1} - w^k}{s^{k+1}}}\
      \pm \sqrt{\Delta_k}}{\norm{u^{k+1}}^2 + \norm{s^{k+1}}^2}, 
\end{align*}
or equivalently,
\begin{equation}\label{eq:rhok-range-ineq}
    \frac{\norm{u^{k+1}}^2 - \tfrac{1}{c_k}\abs{\inner{x^{k+1} - w^k}{s^{k+1}}} - \sqrt{\Delta_k}}{\norm{u^{k+1}}^2 + \norm{s^{k+1}}^2} \leq \rho_k \leq \frac{\norm{u^{k+1}}^2 - \tfrac{1}{c_k}\abs{\inner{x^{k+1} - w^k}{s^{k+1}}} + \sqrt{\Delta_k}}{\norm{u^{k+1}}^2 + \norm{s^{k+1}}^2}.
\end{equation}
If there exist any values of $\rho_k$ in this range that also
meet~\eqref{eq:alt-frame-eps}, one may select any such $\rho_k$ and proceed to
the updates~\eqref{eq:alt-frame-w}-\eqref{eq:alt-frame-p} and the next
iteration.  Otherwise, the subproblem must be solved to higher accuracy. We
will establish below that sufficiently accurate subproblem solutions will lead
to $\Delta_k \geq 0$ and also allow~\eqref{eq:alt-frame-eps} to be satisfied,
except in some cases in which the solving the current subproblem leads to
direct solution of the original primal problem~\eqref{eq:parametric-primal}.
For the time being, however, we will simply assume that it will be possible to
satisfy $\Delta_k \geq 0$, \eqref{eq:rhok-range-ineq},
and~\eqref{eq:alt-frame-eps}.

Such a procedure dynamically adapts the relaxation factors
$\rho_k$ to the accuracy of the subproblem solutions.  An approximate
subproblem solution that might not be acceptable with $\rho_k = 1$, for
example, might be acceptable for a smaller value of $\rho_k$, effectively
under-relaxing the updates applied to $w^k$ and $p^k$.

The quantities $\smallnorm{u^{k+1}}^2$, $\smallnorm{s^{k+1}}^2$, and
$(1/c_k)\abs{\inner{x^{k+1}-w^k}{s^{k+1}}}$ appear repeatedly
in~\eqref{eq:defDeltak} and~\eqref{eq:rhok-range-ineq}. To simplify these
conditions, one may define for all $k \geq 0$ the scalars
\begin{align}
U_k &\doteq \norm{u^{k+1}}^2 &
A_k &\doteq \frac{1}{c_k}\abs{\inner{x^{k+1}-w^k}{s^{k+1}}} &
S_k &\doteq \norm{s^{k+1}}^2,
\end{align}
after which~\eqref{eq:defDeltak} may be expressed as
\begin{equation}
\Delta_k = (U_k - A_k)^2 - \epsilon U_k(U_k + S_k),
\end{equation}
while~\eqref{eq:rhok-range-ineq} becomes
\begin{equation}
\frac{U_k - A_k - \sqrt{\Delta_k}}{U_k + S_k}
\leq \rho_k \leq
\frac{U_k - A_k + \sqrt{\Delta_k}}{U_k + S_k}.
\end{equation}
Equivalently to this last condition, one may select an arbitrary $\gamma_k \in
[-1,1]$ and then write
\begin{align} \label{eq:condensed-rhok-formula}
-1 &\leq \gamma_k \leq 1 &
\rho_k &= \frac{U_k - A_k + \gamma_k\sqrt{\Delta_k}}{U_k + S_k}.
\end{align}
If $u^{k+1} = 0$ and $s^{k+1} = 0$, meaning that $U_k = S_k = 0$,
neither~\eqref{eq:rhok-range-ineq} nor~\eqref{eq:condensed-rhok-formula} is
well defined.  However, in this case~\eqref{eq:alt-frame-min} takes the form
\begin{align*}
(0,0) &\in \partial F(x^{k+1},0) + (0,0 - p^k)
&& \Leftrightarrow & (0,p^k) \in \partial F(x^{k+1},0) &
&& \Leftrightarrow & (0,0) \in \partial L(x^{k+1},p^k),
\end{align*}
the last equivalence arising from~\eqref{eq:subgradient-relations}.  In this
case, $x^{k+1}$ solves the primal problem and $p^k$ solves the dual, so the
algorithm may be halted.  Note also in this case that $\Delta_k = (0 - A_k)^2
- \epsilon \cdot 0 \cdot (0 + 0) = A_k^2 \geq 0$, meaning that $\Delta_k \geq
  0$ must hold.

We now formally restate the
recursions~\eqref{eq:alt-frame-min}-\eqref{eq:alt-frame-p} using the auxiliary
sequences $\{U_k\}$, $\{S_k\}$, and $\{A_k\}$.  Note that the conditions $U_k
> 0$ in~\eqref{eq:alt-frame-u-again}, $\Delta_k \geq 0$
in~\eqref{eq:alt-frame-delta-again}, and $\rho_k > 0$
in~\eqref{eq:alt-frame-rho-again} should be regarded as assumptions for the time
being; we will justify these assumptions below under suitable conditions on
how the method is implemented in practice.

\begin{proposition} \label{prop:convergeUSA}
Suppose that $F:\real^n\times \real^m \to \real \cup \{+\infty\}$ is closed proper convex, $\epsilon \in (0,1)$, and the sequences
$\{x^k\}^\infty_{k=1}, \{w^k\}^\infty_{k=0}, \{s^k\}^\infty_{k=1}\subset
\real^{n}$, $\{p^k\}^\infty_{k=0}, \{\bar p^k\}^\infty_{k=1},
\{u^k\}^\infty_{k=1}\subset \realm$, $\{\rho_k\}^\infty_{k=0}$, and
$\{U_k\}^\infty_{k=0}, \{S_k\}^\infty_{k=0}, \{A_k\}^\infty_{k=0},
\{\gamma_k\}, \{c_k\}^\infty_{k=0}\subset \real$, with $\inf_k\{c_k\} > 0$,
obey the following recursions for all $k
\geq 0$:
\begin{align}
    (s^{k+1}, 0)
    & \in \partial F(x^{k+1}, u^{k+1}) + (0,c_k u^{k+1} - p^k) 
    \label{eq:alt-frame-min-again} \\
    U_k &= \norm{u^{k+1}}^2 > 0 \label{eq:alt-frame-u-again} \\ 
    S_k &= \norm{s^{k+1}}^2 \label{eq:alt-frame-s-again} \\ 
    A_k &= \frac{1}{c_k}{\abs{\inner{x^{k+1}-w^k}{s^{k+1}}}}\label{eq:alt-frame-a-again} \\ 
    \Delta_k &= (U_k - A_k)^2 - \epsilon U_k(U_k + S_k) \geq 0 \label{eq:alt-frame-delta-again} \\ 
    \gamma_k &\in [-1,1] \label{eq:alt-frame-gamma-again} \\ 
    \rho_k &= \frac{U_k - A_k + \gamma_k\sqrt{\Delta_k}}{U_k + S_k} > 0 \label{eq:alt-frame-rho-again} \\ 
    w^{k+1} & = w^k - \rho_k c_k s^{k+1} \label{eq:alt-frame-w-again} \\
    \bar{p}^k & = p^k - c_k u^{k+1} \label{eq:alt-frame-bar-p-again} \\
    p^{k+1} & = (1-\rho_k) p^k + \rho_k \bar p^k \label{eq:alt-frame-p-again}
\end{align}
Then the sequences $\{x^k\}$, $\{w^k\}$, $\{s^k\}$, $\{p^k\}$, $\{u^k\}$,
$\{\rho_k\}$, and $\{c_k\}$ conform to the
recursions~\eqref{eq:alt-frame-min}-\eqref{eq:alt-frame-p}.  In particular, if
a saddle point of $L$ as defined in~\eqref{eq:def-lagrange} exists, then all
the conclusions of Proposition~\ref{prop:convergence} hold, and otherwise at
least one of $\{x^k\}$ and $\{\bar p^k\}$ is unbounded.
\end{proposition}
\begin{proof}
The inclusion~\eqref{eq:alt-frame-min-again} is identical
to~\eqref{eq:alt-frame-min},
while~\eqref{eq:alt-frame-w-again}-\eqref{eq:alt-frame-p-again} are identical
to~\eqref{eq:alt-frame-w}-\eqref{eq:alt-frame-p}, so to prove the claim
about~\eqref{eq:alt-frame-min-again}-\eqref{eq:alt-frame-p-again} producing
sequences conforming to~\eqref{eq:alt-frame-min}-\eqref{eq:alt-frame-p}, it
will suffice to show, for any $k \geq 0$,
that~\eqref{eq:alt-frame-u-again}-\eqref{eq:alt-frame-rho-again}
imply~\eqref{eq:alt-frame-approx} and~\eqref{eq:alt-frame-eps}.  To this end,
fix any $k\geq 0$, and observe that by substituting the definitions of $U_k$,
$S_k$ and $A_k$ into~\eqref{eq:alt-frame-approx} one obtains the equivalent
inequality
\begin{equation} \label{approxUSA}
2\rho_k A_k + \rho_k^2 S_k \leq (2\rho_k - \rho_k^2 - \epsilon) U_k.
\end{equation}
Algebraically rearranging, this inequality is in turn equivalent to
\begin{equation} \label{approxUSA2}
(S_k + U_k)\rho_k^2 + 2(A_k - U_k) \rho_k + \epsilon U_k \leq 0.
\end{equation}
The left side of this inequality is a quadratic function of $\rho_k$ whose
discriminant is
\begin{equation*}
\big(2(A_k-U_k)\big)^2 - 4(U_k + S_k)(\epsilon U_k) 
= 4(A_k-U_k)^2 - 4\epsilon U_k (U_k + S_k) = 4 \Delta_k.
\end{equation*}
Since $\Delta_k \geq 0$ by assumption, there exist values of $\rho_k$
satisfying~\eqref{approxUSA2} has solutions, which the quadratic formula
dictates consist  of all numbers between the two values
\begin{equation*}
\frac{-2(A_k-U_k) \pm \sqrt{4\Delta_k}}{2(U_k + S_k)}
= \frac{U_k-A_k \pm \sqrt{\Delta_k}}{U_k + S_k}.
\end{equation*}
This is precisely the range of values described
by~\eqref{eq:alt-frame-gamma-again} and the equation
in~\eqref{eq:alt-frame-rho-again}, hence~\eqref{approxUSA2} holds, which is in
turn equivalent to~\eqref{approxUSA} and~\eqref{eq:alt-frame-approx}.
So~\eqref{eq:alt-frame-approx} holds and it remains to show
that~\eqref{eq:alt-frame-eps} is also true.  To this end,
consider~\eqref{approxUSA}: since $\rho_k > 0$ by assumption
in~\eqref{eq:alt-frame-rho-again} and both $A_k$ and $S_k$ are necessarily
nonnegative, the left side of the inequality 
must be nonnegative.  On the right, $U_k > 0$ by
assumption in~\eqref{eq:alt-frame-u-again}, so necessarily $2\rho_k - \rho_k^2
- \epsilon \geq 0$ since otherwise the right side would be negative and a
  contradiction would immediately ensue. Thus,
  \eqref{eq:alt-frame-eps} holds, completing the demonstration that the
  sequences conform to~\eqref{eq:alt-frame-min}-\eqref{eq:alt-frame-p}.

In the case the $L$ has at least one saddle point, all of the hypotheses of
Proposition~\ref{prop:convergence} therefore hold, and so all its conclusions
hold.  Similarly, if $L$ has no saddle points,
Proposition~\ref{prop:unbounded} may be invoked to assert that at least one of
$\{x^k\}$ and $\{\bar p^k\}$ is unbounded.
\end{proof}

\subsection{A complete abstract algorithmic framework}
The assumptions $U_k > 0$, $\Delta_k \geq 0$, and $\rho_k > 0$ respectively
in~\eqref{eq:alt-frame-min-again}, \eqref{eq:alt-frame-delta-again},
and~\eqref{eq:alt-frame-rho-again} made it possible to demonstrate that the
recursion
framework~\eqref{eq:alt-frame-min-again}-\eqref{eq:alt-frame-p-again}
describes sequences conforming to the originally proposed recursion
framework~\eqref{eq:alt-frame-min}-\eqref{eq:alt-frame-p}.  This section will
next establish that, if one is equipped with a plausible iterative procedure
for the approximate minimization~\eqref{eq:alt-frame-min-again} --- or
equivalently~\eqref{eq:alt-frame-min} or~\eqref{eq:saddle-point-l} --- then it
is possible to either guarantee that these assumptions hold or that the
process of solving~\eqref{eq:alt-frame-min-again} will provide a direct
solution to the original primal and dual problems.  This task will be
accomplished by proposing a complete version of a full algorithmic framework
using the results in the preceding sections.  Furthermore, so that the
analysis remains general, we will first posit an abstract model for whatever
iterative procedure might be used to solve the approximate minimization
problem~\eqref{eq:alt-frame-min-again}:

\begin{definition} \label{def:subSolver}
Given a closed proper convex function $F:\real^n \times \real^m \to \real \cup
\{+\infty\}$, a vector $p^k \in
\real^m$, and a scalar $c_k > 0$; a \emph{compatible subproblem process} for 
$(F,p^k,c_k)$ is a sequence
$\big\{(x^{k,j},u^{k,j},s^{k,j})\big\}_{j=0}^\infty$ such that
\begin{align}
\lim_{j\to\infty} s^{k,j} &= 0 \label{CSPconverge} \\
\lim_{j\to\infty} \inner{x^{k,j}}{s^{k,j}} &= 0 \label{CSPinner} \\
(\forall\,j \geq 0) \qquad \qquad
(s^{k,j}, 0) & \in \partial F(x^{k,j}, u^{k,j}) + (0,c_k u^{k,j} - p^k).
    \label{CSPsubgrad}
\end{align}
\end{definition}

Here, we use $j$ as an iteration counter for the ``inner loop'' subproblem of minimizing the augmented Lagrangian minimand $F_k(x,u) \doteq F(x,u) - \inner{p^k}{u}
+ (c_k/2)\smallnorm{u}^2$ in~\eqref{generalAugLagMin}. The index $k$ is
  essentially superfluous to the definition, but is included to make it easier
  to match the definition's notation to our later analysis.
  The definition abstracts an iterative process that attempts
  to converge to a minimizer of $F_k$ by producing a sequence of iterates
  $\big\{(x^{k,j},u^{k,j})\big\}_{j=0}^\infty$ and corresponding subgradients
  $\big\{(s^{k,j},0)\big\}_{j=0}^\infty$.  The condition~\eqref{CSPsubgrad}
  specifies that $(s^{k,j}, 0) \in \partial F_k(x^{k,j}, u^{k,j})$ for all
  $j$. The subgradient with respect to the $u$ variables is always zero,
  meaning that one always minimizes exactly with respect to $u$,
  while~\eqref{CSPconverge} specifies that the $x$ subgradient components are
  driven to zero as $j \to \infty$, reflecting convergence to the optimality
  conditions $(0,0) \in \partial F_k(x,u) = \partial\big[F(x,u) -
  \inner{p^k}{u} + (c_k/2)\smallnorm{u}^2\big]$.  The reasons for the inner
  product limit condition~\eqref{CSPinner} is more technical and may be seen in
  the proof of Lemma~\ref{lem:genericInnerLoop} in 
  Appendix~\ref{app:inner-loop-proof}, but it is
  readily guaranteed, in view of~\eqref{CSPconverge}, since $\lim_{j\to\infty}
  s^{k,j} \to 0$ by~\eqref{CSPconverge}, by $\{x^{k,j}\}_{j=0}^\infty$ being
  convergent or bounded.

We now state a somewhat less abstract
generic algorithm able to
implement~\eqref{eq:alt-frame-min-again}-\eqref{eq:alt-frame-p-again}:

~

\begin{algorithm}[ht]{}
    \caption{Generic Adaptive Relaxation Algorithm} \label{alg:gen-relax}
    \begin{algorithmic}[1]
        \State \textbf{Given:} $p^0 \in \real^m$, $w^{0}=0 \in \real^m$, 
            $c_{\min} > 0$, $\epsilon \in (0,1)$
        \For {$k=0,1,2, \cdots$} \algorithmiccomment{Outer loop}
        \State Choose $c_k>c_{\min}$
        \State \textbf{Given:} $\big\{(x^{k,j},u^{k,j},s^{k,j})\big\}_{j=0}^\infty$ 
                is a compatible subproblem process for $(F,p^k,c_k)$
        \Repeat{~for $j=1, 2, \ldots$} \algorithmiccomment{Inner loop}
           \State $U_{k,j} = \norm{u^{k,j}}^2$ \label{step:calcUkj}
           \State $S_{k,j} = \norm{s^{k,j}}^2$
           \State $A_{k,j} = \frac{1}{c_k}\abs{\inner{x^{k,j}-w^k}{s^{k,j}}}$
                  \label{step:calcAkj}
           \State $\Delta_{k,j}= (U_{k,j} - A_{k,j})^2 
                                        - \epsilon U_{k,j} (U_{k,j} + S_{k,j})$
                  \label{step:calcDeltakj}
        \Until{$A_{k,j} < U_{k,j}$ and $\Delta_{k,j} \geq 0$}
              \label{step:genericInnerStop}
        \State $x^{k+1} = x^{k,j}$ \quad $u^{k+1} = u^{k,j}$ \quad 
               $s^{k+1} = s^{k,j}$ \label{step:jfreeze1}
        \State $U_k = U_{k,j}$ \quad $S_k = S_{k,j}$ \quad $A_k = A_{k,j}$ 
                               \quad $\Delta_k = \Delta_{k,j}$ 
                               \label{step:jfreeze2}
        \State $\gamma_k^{\min} = \max\left\{-1, \frac{A_k - U_k}{\sqrt{\Delta_k}}
                                           \right\}$ \label{step:calcgammakmin}
        \State Choose any $\gamma_k \in (\gamma_k^{\min},1]$
               \label{step:choosegammak}
        \State $\rho_k = \frac{U_k - A_k + \gamma_k\sqrt{\Delta_k}}{U_k + S_k}$
               \label{step:calcrhok}
        \State $w^{k+1} = w^k - \rho_k c_k s^{k+1}$ \label{step:wupdate}
        \State $\bar{p}^k = p^k - c_k u^{k+1}$ 
        \State $p^{k+1} = (1-\rho_k) p^k + \rho_k \bar p^k$ \label{step:pupdate}
    \EndFor
    \end{algorithmic}
\end{algorithm}

~

With the help of our earlier results, we now give a convergence analysis of
this algorithm.  Technical results regarding becoming indefinitely ``stuck''
in the inner loop in steps~\ref{step:calcUkj}-\ref{step:genericInnerStop}, a
situation we never observed in our computational experiments in
Section~\ref{sec:numerical-experiments}, are deferred to appendices.

\begin{lemma} \label{lem:genericInnerLoop}
Suppose for some $k\geq 0$ that the inner loop of
Algorithm~\ref{alg:gen-relax} runs indefinitely, never satisfying the
stopping condition in Step~\ref{step:genericInnerStop}.  Then
\begin{align*}
  \lim_{j\to\infty} U_{k,j} &= 0 &
  \lim_{j\to\infty} S_{k,j} &= 0 &
  \lim_{j\to\infty} A_{k,j} &= 0.
\end{align*}
\end{lemma}
\begin{proof}
See~\ref{app:inner-loop-proof}.
\end{proof}

\begin{proposition} \label{prop:fullAbstract}
Algorithm~\ref{alg:gen-relax} must produce one of the following two results:
\begin{enumerate}[label=(\roman*), nosep]
\item \label{item:notstuck} Every instance of the inner loop terminates finitely, 
so the outer 
loop runs indefinitely.  In this case, all the conclusions of
Proposition~\ref{prop:convergence} hold if a saddle point of $L$ exists,
and otherwise at least one of the sequences  $\{x^k\}$ or $\{\bar p^k\}$
produced by the outer loop must be unbounded.
\item \label{item:stuckInner} For some $k$, the inner loop runs indefinitely.
In this case, $p^k$ is an optimal solution to the dual problem and the
sequence $\{x^{k,j}\}_{j=0}^\infty$ is asymptotically primal optimal in the
sense that
\begin{align*}
    \limsup_{j\to\infty} F(x^{k,j}, u^{k,j}) &\leq \inf_{x\in\real^n} F(x,0) &
    \lim_{j\to\infty} u^{k,j} &= 0.
\end{align*}
\end{enumerate}
\end{proposition}
\begin{proof}
First consider case~\ref{item:notstuck}, in which the outer loop runs
indefinitely.  Fixing any $k \geq 0$, our goal will now be to show that all
the conditions in Proposition~\ref{prop:convergeUSA} hold.  The inner loop
termination conditions and immediately subsequent assignments in
steps~\ref{step:genericInnerStop}-\ref{step:jfreeze2} imply that $A_k < U_k$
and $\Delta_k \geq 0$.  Since $A_k \geq 0$ (as may be seen by examining
steps~\ref{step:calcAkj} and~\ref{step:jfreeze2} of the algorithm), it
immediately follows that $U_k > A_k \geq 0$.  Thus, the assumptions $U_k > 0$
and $\Delta_k \geq 0$ in Proposition~\ref{prop:convergeUSA} are satisfied.
From the compatible subproblem process definition and the calculations in
steps~\ref{step:calcUkj}-\ref{step:calcDeltakj}, it is now clear that all the
remaining conditions
in~\eqref{eq:alt-frame-min-again}-\eqref{eq:alt-frame-delta-again} hold.

Next, we claim that $\gamma_k \in [-1,1]$ and $\rho_k$ in
step~\ref{step:calcrhok} meets the conditions
in~\eqref{eq:alt-frame-gamma-again}-\eqref{eq:alt-frame-rho-again}.  The form
of the calculation in step~\ref{step:calcrhok} exactly matches the equation
in~\eqref{eq:alt-frame-rho-again}, so showing that $\gamma_k \in [-1,1]$ and
$\rho_k > 0$ will be sufficient to establish the claim.  Now, since $A_k <
U_k$, the value $(A_k - U_k)/\sqrt{\Delta_k}$ calculated in
step~\ref{step:genericInnerStop} is negative, and $-1 \leq \gamma_k^{\min}
< 0$ by construction.  This means that the range within which $\gamma_k$
may be selected in step~\ref{step:choosegammak} is nonempty.  Again by construction,
\begin{equation*}
  \gamma_k > \gamma_k^{\min} \geq \frac{A_k - U_k}{\sqrt{\Delta_k}},
\end{equation*}
and hence the numerator in the immediately ensuing calculation of $\rho_k$ is
\begin{align*}
  U_k - A_k + \gamma_k \sqrt{\Delta_k}
  > U_k - A_k + \left(\frac{A_k - U_k}{\sqrt{\Delta_k}}\right) \sqrt{\Delta_k}
  = 0.
\end{align*}
Furthermore, the denominator $U_k + S_k$ in the calculation of $\rho_k$ must
be positive since $U_k > A_k \geq 0$ and $S_k \geq 0$ by
steps~\ref{step:calcUkj}-\ref{step:calcAkj} and~\ref{step:genericInnerStop},
so one may conclude that $\rho_k > 0$, satisfying the restriction on $\rho_k$
in~\eqref{eq:alt-frame-rho-again}. Finally,
steps~\ref{step:wupdate}-\ref{step:pupdate} exactly match
conditions~\eqref{eq:alt-frame-w-again}-\eqref{eq:alt-frame-p-again}, so all
of~\eqref{eq:alt-frame-min-again}-\eqref{eq:alt-frame-p-again} hold.  Since
the choice of $k$ was arbitrary, the entire recursion system holds for all $k$
and we may invoke Proposition~\ref{prop:convergeUSA} to guarantee that all the
conclusions of~Proposition~\ref{prop:convergence} hold if saddle point of $L$
exists and at least one of the sequences $\{x^k\}$ or $\{\bar p^k\}$ is
unbounded if one does not.

It remains to consider case~\ref{item:stuckInner}, when the inner loop
continues indefinitely for some value of the outer loop index $k$.  The
proof for this case is in \ref{app:full-generic-stuck-inner}.
\end{proof}

The analysis of case~\ref{item:stuckInner} of the above proposition, in which
the method gets ``stuck'' in an infinite inner loop, does not require the
existence of a saddle point of $L$.  However, if such a saddle point
$(x^*,p^*)$ exists, then it is immediate that
\begin{equation*}
\liminf_{j\to\infty} F(x^{k,j},u^{k,j}) \leq F(x^*, 0) = D(0,p^*) = D(0,p^k).
\end{equation*}

\section{A Connection Between the ADMM and the Proximal Gradient Method}
\label{sec:connection-between-admm}

For the remainder of this paper, we concentrate on Fenchel-Rockafellar problem
class~\eqref{eq:conjugate-primal} under the parameterization~\eqref{eq:sum-F}.
The Fenchel-Rockafellar dual problem of \eqref{eq:primal-problem} is
\begin{equation} \label{eq:dual-problem}
    \min_{p\in \mathbb{R}^m} f^*(-M\transpose p) + g^*(p).
\end{equation}
Now consider applying the proximal point algorithm to the dual
problem~\eqref{eq:dual-problem}.  In its simplest form, applying the proximal
point algorithm to~\eqref{eq:dual-problem} may be written
\begin{equation} \label{eq:dual-ppa}
p^{k+1} = \argmin_{p\in\real^m} 
  \left\{
     f^*(-M\transpose p) + g^*(p) + \frac{1}{2c_k}\norm{p - p^k}^2     
  \right\},
\end{equation}
where $\{c_k\}$ is a sequence of positive scalars bounded away from zero, as
in the preceding sections. By a standard duality analysis following the ideas
first laid out in~\cite{Roc76b}, it can be shown that~\eqref{eq:dual-ppa} is
equivalent to the augmented Lagrangian
method~\eqref{almfenchelmin}-\eqref{almfenchelupdate}. Furthermore, relaxation
of the multiplier adjustment step, meaning that one adjusts the multipliers by
$\rho_k c_k(Mx^{k+1} - z^{k+1})$ instead of just $c_k(Mx^{k+1}
- z^{k+1})$, is equivalent to applying the relaxation factor $\rho_k$ to the dual
formulation of the algorithm~\eqref{eq:dual-ppa}.

Now consider an alternative formulation of~\eqref{eq:dual-ppa}:
after regrouping, the subproblem at each step may be written
\begin{equation} \label{eq:dual-subprob}
\min_{p\in\real^m} \left\{
    \left(  \big( f^* \circ \big(-M\transpose  \big) \big)(\cdot) + \frac{1}{2c_k}\| \cdot - p^k\|^2 \right)\!(p) + g^*(p) \right\}.
\end{equation}
Consider next the Fenchel dual of~\eqref{eq:dual-subprob}, whose
objective function is
\begin{multline} \label{eq:split-subprob}
    \left( \big( f^* \circ \big(-M\transpose  \big) \big)(\cdot) + \frac{1}{2c_k}\| \cdot - p^k\|^2 \right)^{*} \!\! (-z) + g^{**}(z) \\
    =  \left( \big( f^* \circ \big(-M\transpose  \big) \big)(\cdot) + \frac{1}{2c_k}\| \cdot - p^k\|^2 \right)^{*} \!\!(-z) + g(z),
\end{multline}
since $g^{**}(z) = g(z)$ by the standard conjugacy properties of closed convex
functions. For all $k \geq 0$, we adopt the notation
\begin{align} \label{eq:hkfk}
    h_k(p) &\doteq f^*(-M\transpose p) + \frac{1}{2c_k}\| p - p^k\|^2 & 
    f_k(z) &\doteq h_k^*(-z). %
\end{align}
Then the dual of the subproblem being solved in each step~\eqref{eq:dual-ppa} is
\begin{equation} \label{eq:defbal-prob}
    \min_{z\in\real^m} \big\{ f_k(z) + g(z) \big\}.
\end{equation}
Since $h_k$ is $(1/c_k)$-strongly convex, its conjugate $f_k$ is
differentiable, with a $c_k$-Lipschitz gradient. If the proximal mapping of
$g$ may be readily evaluated --- that is, $g$ is ``prox friendly'' --- it is
thus attractive to consider applying the proximal gradient (forward-backward)
method to the minimization in~\eqref{eq:defbal-prob}.  The same interaction
between strong convexity, duality, and the forward-backward splitting has been
used in a different context~\cite{CDV10} (but outside the context of an outer
augmented Lagrangian method, and with the roles of the primal and dual
reversed), and generalizations of this work appear in~\cite{CDV11,CV14}.  
This material is revisited within a general framework in~\cite{Com24}.

Applying the proximal gradient method involves first executing a forward
(gradient) step to $f_k$, followed by a proximal step on $g$. Although the
definition of $f_k$ may appear complicated, its gradient may be computed by
solving a relatively straightforward minimization problem, as shown in the
following lemma:

\begin{lemma} \label{lemma:derivative-fk}
    For any $k \geq 0$ and $z\in\real^m$, suppose that
    \begin{equation} \label{eq:x-min}
        \bar x \in \argmin_{x\in\real^n} 
           \left\{ f(x) + \inner{p^k}{Mx}  + \frac{c_k}{2}\norm{Mx - z}^2  \right\}. 
    \end{equation}
    Then $\nabla f_k(z) = -\big(p^k + c_k(M\bar x-z)\big)$.
\end{lemma}
\begin{proof}
See~\ref{app:derivative-fk-proof}.
\end{proof}
Now consider, at any step $k\geq 0$ of proximal point
algorithm~\eqref{eq:dual-ppa}, employing the proximal gradient method to solve
the dual subproblem~\eqref{eq:defbal-prob}.  Let $j$ represents the ``inner,''
proximal-gradient iteration counter within each subproblem, and denote the
iterates of the proximal gradient sub-algorithm by $z^{k,j}$ and its stepsizes
by $\alpha_{k,j} > 0$. The subproblem solution procedure may be expressed as
\begin{equation} \label{eq:abstract-subproblem-fb}
z^{k,j+1} 
  = \prox_{(\alpha_{k,j}) g}\!\big(z^{k,j} - \alpha_{k,j}\nabla f_k(z^{k,j})\big).
\end{equation}
By Lemma~\ref{lemma:derivative-fk}, if one has
\begin{equation} \label{eq:admm-like-xkj}
 x^{k,j+1} \in \argmin_{x\in\real^n} 
           \left\{ f(x) + \inner{p^k}{Mx}  + \frac{c_k}{2}\norm{Mx - z^{k,j}}^2  \right\}, 
\end{equation}
then $\nabla f_k(z^{k,j}) = -\big(p^k + c_k(Mx^{k,j+1}-z^{k,j})\big)$. Substituting
this formula into~\eqref{eq:abstract-subproblem-fb}, one obtains
\begin{align*}
    z^{k,j+1} & = \prox_{(\alpha_{k,j}) g}
                     \left(z^{k,j} + \alpha_{k,j}\big(p^k + c_k(Mx^{k,j+1}-z^{k,j})
                     \big)\right) \\
              & = \prox_{(\alpha_{k,j}) g}\left(c_k \alpha_{k, j} M x^{k, j+1}
                  +\left(1-c_k \alpha_{k, j}\right) z^{k, j}+\alpha_{k, j} p^k \right).
\end{align*}
By the definition $\prox_{c g}(w) \doteq \argmin_{z \in
\mathbb{R}^n}\big\{g(z)+\frac{1}{2 c}\smallnorm{z-w}^2\big\}$, one then obtains
\begin{align}
    z^{k, j+1} & = \argmin_{z\in\real^m}\left\{g(z)+\frac{1}{2 \alpha_{k, j}}\left\|z-\left(c_k \alpha_{k, j} M x^{k, j+1}+\left(1-c_k \alpha_{k, j}\right) z^{k, j}+\alpha_{k, j} p^k\right)\right\|^2\right\} \nonumber \\
    & =\argmin_{z\in\real^m}\left\{g(z)-\left\langle p^k, z\right\rangle+\frac{1}{2 \alpha_{k, j}}\left\|c_k \alpha_{k, j} M x^{k, j+1}+\left(1-c_k \alpha_{k, j}\right) z^{k, j}-z\right\|^2\right\},
    \label{eq:zkj-with-alpha}
\end{align}
where the second equality follows by expanding the square and dropping a
constant term. 

For each $k$ and $j$, define $\nu_{k,j}=c_k\alpha_{k,j}$, so 
$\alpha_{k,j}=\nu_{k,j}/c_k$. One may then express the calculation of $z^{k,j}$ as
\begin{equation*}
    z^{k, j+1}=\argmin_{z\in\real^m}\left\{g(z)-\left\langle p^k, z\right\rangle+\frac{c_k}{2 v_{k, j}}\left\|\left(v_{k, j} M x^{k, j+1}+\left(1-v_{k, j}\right) z^{k, j}\right)-z\right\|^2\right\}.
\end{equation*}
Recalling the calculation of $x^{k,j+1}$ from~\eqref{eq:admm-like-xkj}, we
conclude that the proximal gradient method~\eqref{eq:abstract-subproblem-fb}
may be performed through the recursions,
\begin{align}
x^{k,j+1} &\in \argmin_{x\in\real^n} 
           \left\{ f(x) + \inner{p^k}{Mx}  
             + \frac{c_k}{2}\norm{Mx - z^{k,j}}^2  \right\} \label{eq:admm-xkj-again} \\
    z^{k, j+1} &=\argmin_{x\in\real^n}\left\{g(z)-\left\langle p^k, z\right\rangle+\frac{c_k}{2 v_{k, j}}\left\|\left(v_{k, j} M x^{k, j+1}+\left(1-v_{k, j}\right) z^{k, j}\right)-z\right\|^2\right\}, \label{eq:admm-zkj}
\end{align}
for some positive scalars $\nu_{k,j}$, provided that~\eqref{eq:admm-xkj-again}
can be solved.

Since $\nabla f_k$ is $c_k$-Lipschitz continuous, natural choices of the
proximal gradient stepsizes $\alpha_{k,j}$ satisfy $0 < \inf_j
\{\alpha_{j,k}\}$ and $\sup_j \{\alpha_{j,k}\} < 2/c_k$, which is equivalent to 
having $0 < \inf_j\{\nu_{k,j}\}$ and $\sup_j \{\nu_{j,k}\} < 2$.  If one
chooses the midpoint of this range, $\nu_{k,j} = 1$, the above recursions
simplify to
\begin{align}
& x^{k, j+1} \in \argmin_{x\in\real^n}\left\{f(x)+\left\langle p^k, M x\right\rangle+\frac{c_k}{2}\left\|M x-z^{k, j}\right\|^2\right\} \label{eq:admm-like-x-nukj1} \\
& z^{k, j+1}=\argmin_{x\in\real^n}\left\{g(z)-\left\langle p^k, z\right\rangle+\frac{c_k}{2}\left\|M x^{k, j+1}-z\right\|^2\right\},  \label{eq:admm-like-z-nukj1}
\end{align}
Other than the indexing scheme, these recursion are identical to the first two
steps of the ADMM for solving the problem~\eqref{eq:primal-problem}.  

Of course $\nu_{j,k} = 1$ is not the only possible choice, and the natural
range of $\nu_{j,}$ implied by the theory of the proximal gradient method is
between $0$ and $2$ (and bounded strictly away from each).  With these bounds,
\eqref{eq:admm-xkj-again}-\eqref{eq:admm-zkj} strongly resemble the first two
steps of the generalized ADMM first introduced in~\cite{EckBer92}, the only
difference being the presence of $\nu_{k,j}$ in the denominator
in~\eqref{eq:admm-zkj}.

As mentioned in the introduction, this alternating direction recursion for
solving the subproblem~\eqref{eq:dual-ppa} was suggested as long ago as
1975~\cite{GM75}.  As also pointed out in the introduction, prior theoretical
results regarding its convergence were not entirely satisfactory; given its
equivalence to~\eqref{eq:abstract-subproblem-fb}, its convergence can be
inferred from the behavior of the proximal gradient method.  We now make this
linkage in the course of demonstrating that the generalization of alternating
minimization in~\eqref{eq:admm-like-x-nukj1}-\eqref{eq:admm-like-z-nukj1}
yields a compatible subproblem process, corresponding to $\nu_k \equiv 1$.
More general choices of $\nu_k$ are possible, but for brevity we only consider
the simple case~\eqref{eq:admm-like-x-nukj1}-\eqref{eq:admm-like-z-nukj1}
here.

\begin{proposition} \label{prop:alternatingCSP}
Consider the Fenchel-Rockafellar problem~\eqref{eq:conjugate-primal} under the
parameterization $F$ in~\eqref{eq:sum-F}. Fix any $k\geq 0$, $c_k\geq 0$, and
$p^k\in\real^m$, and assume that the corresponding augmented Lagrangian
subproblem~\eqref{eq:defbal-prob} has a solution. Then, for any $z^{k,0}
\in \real^m$, any sequence $\{x^{k,j}\}$ and $\{z^{k,j}\}$ conforming to the
recursions~\eqref{eq:admm-like-x-nukj1}-\eqref{eq:admm-like-z-nukj1}, together with the
sequence $\{s^{k,j}\}\in\real^n \times \real^n$ defined by
\begin{equation}\label{subgrad-fb}
(\forall\,j \geq 0) \qquad
    s^{k,j+1} \doteq 
    \begin{pmatrix}
        s_x^{k,j+1} \\
        0
    \end{pmatrix} \doteq
    \begin{pmatrix}
        c_k M\transpose (z^{k,j} - z^{k,j+1}) \\
        0
    \end{pmatrix},
\end{equation}
constitutes a compatible subproblem process for $(F,p^k,c_k)$.
\end{proposition}
\begin{proof}
See~\ref{app:alternatingCSP-proof}.
\end{proof}

In brief, the above proposition states that alternating minimization is
guaranteeably usable subproblem solver for applying inexact augmented
Lagrangian -- such as the ones analyzed here and the simpler case covered
in~\cite{eckstein2013practical} --- to problems in Fenchel-Rockafellar form.
However, the insights uncovered earlier in this section imply that there may
be faster ways to use the same basic operations to minimize the augmented
Lagrangian, as we will now discuss.

\section{An ADMM-Like Framework}
\label{sec:admm-like-framework}
The above proposition provides a rigorous justification for using the
iteration~\eqref{eq:admm-like-x-nukj1}-\eqref{eq:admm-like-z-nukj1} to perform
the augmented Lagrangian minimization~\eqref{almfenchelmin}.  Before, the best
available result to justify this procedure appeared to be a subsequential
convergence result due to Tseng~\cite{Tse01}, but this result was designed for
considerably more general situations and did not take full advantage of the
duality structure of~\eqref{almfenchelmin}.  By itself, however, the result is
not of great practical interest since running many iterations
of~\eqref{eq:admm-like-x-nukj1}-\eqref{eq:admm-like-z-nukj1} to attempt to
minimize the augmented Lagrangian can be very inefficient, as noted
in~\cite{PJO}.

However, the insight that applying alternating minimization to
solve~\eqref{almfenchelmin} is equivalent to a dual application of the
proximal gradient method opens a range of new algorithmic variations beyond
those examined in~\cite{PJO}:
\begin{itemize}
\item One may use any of the available approximation criteria to decide when an approximate
solution to the subproblem is acceptable.  In particular, we consider using
the new approximate relaxed augmented Lagrangian method of
Section~\ref{sec:extended-criterion}, adjusting the relaxation of the
multiplier update step to match the accuracy of the subproblem solution.  This
strategy had not been developed when~\cite{PJO} was written.

\item Instead of applying the standard proximal gradient method to the dual of
each subproblem, one may instead employ essentially any convergent method that
minimizes the sum of two convex functions through gradient steps on one
function and proximal steps on the other.  Specifically, rather than use a
``vanilla'' proximal gradient method, one may use an ``accelerated'' method
such as the FISTA and related methods described in~\cite{beck2009fast}
and~\cite{chambolle2015convergence}.  The resulting ``inner loops'' that
approximately solve the subproblems are somewhat more complicated
that~\eqref{eq:admm-xkj-again}-\eqref{eq:admm-zkj} and its special
case~\eqref{eq:admm-like-x-nukj1}-\eqref{eq:admm-like-z-nukj1}, but still use
the same basic operations.
\end{itemize}

As in Section~\ref{sec:connection-between-admm}, our notation uses $k$ to
index the steps of the ``outer'', inexact augmented Lagrangian method, while
$j$ indexes the steps of the inner, proximal-gradient-class method that
approximately solves the augmented Lagrangian subproblems. This ``inner''
method progressively improves the $x^{k,j+1}$ and $z^{k,j+1}$ solutions by
minimizations respectively involving $f$ and $g$, until a stopping criterion
is met. The associated subgradient $s^{k,j+1}$ at inner iteration $j+1$ is
given by the equivalent forms
of~\eqref{eq:sum-frame-bar-p}-\eqref{eq:sum-frame-s} stated
in~\eqref{eq:sum-frame-subgrad-x}-\eqref{eq:sum-frame-subgrad-z}.  Including the inner-iteration indices $j$ in these relations produces
\begin{equation} \label{eq:subgradient-lagrange-kj}
    s^{k,j+1} =
    \begin{pmatrix}
        s_x^{k,j+1} \\ s_z^{k,j+1}
    \end{pmatrix} \in
    \begin{pmatrix}
        \partial f(x^{k,j+1}) + M\transpose p^k + c_kM\transpose (Mx^{k,j+1}-z^{k,j+1}) \\
        \partial g(z^{k,j+1}) - p^k - c_k(Mx^{k,j+1}-z^{k,j+1})
    \end{pmatrix},
\end{equation}
Suitable values of $s_x^{k,j+1}$ and $s_z^{k,j+1}$ may be obtained as
byproducts of the $x$ and $z$ minimizations in~\eqref{eq:admm-xkj-again}
and~\eqref{eq:admm-zkj}, so little extra effort is required to obtain them.
Similarly to the decomposition of $s^{k,j+1}$, we also decompose $w^{k}$ into
$w_x^{k}$ and $w_z^{k}$, and the $w^k$ update~\eqref{eq:alt-frame-w}
or~\eqref{eq:sum-frame-w} may be expressed as the two recursions $w_x^{k+1} =
w_x^{k} - \rho_{k} c_{k} s_x^{k}$ and $w_z^{k+1} = w_z^{k} - \rho_{k}
c_{k} s_z^{k}$. Including the inner-loop indices $j$
in~\eqref{eq:sum-frame-approx}, on may express the principal stopping
criterion for the inner loop as
\begin{multline} \label{eq:inner-stopping}
    \frac{2\rho_k}{c_k}
    \abs{\inner{x^{k,j+1} - w_x^k}{s_x^{k,j+1}}} + \frac{2\rho_k}{c_k}\abs{\inner{z^{k,j+1} - w_z^k}{s_z^{k,j+1}}} + \rho_k^2 \norm{s^{k,j+1}}^2
    \\ \leq \parent{2\rho_k - \rho_k^2 - \epsilon} \norm{Mx^{k,j+1}-z^{k,j+1}}^2. 
\end{multline}
If this condition holds at inner iteration $j$, one may set
$x^{k+1}=x^{k,j+1}$, $z^{k+1}=z^{k,j+1}$, and $s^{k+1}=s^{k,j+1}$ to obtain
sequences $\{x^k\}$, $\{z^k\}$, and $\{s^k\}$ that conform to our main
recursions~\eqref{eq:sum-frame-bar-p}-\eqref{eq:sum-frame-p}.

\subsection{FISTA-CD as the Subproblem Solver}
\label{sec:fista-cd}
We now consider using a FISTA-family algorithm to solve the subproblems,
instead of an ordinary forward-backward method (shown in
Section~\ref{sec:connection-between-admm} to be equivalent to alternating
minimization of the augmented Lagrangian). We now state the details of the
algorithm and its convergence results.

The original FISTA algorithm with constant step size was proposed
in~\cite{beck2009fast}. It applies to the sum of a smooth convex function with
$L$-Lipschitz gradient, which in our case it will be $f_k$ as defined
in~\eqref{eq:hkfk}, and a continuous convex function, which in our case it
will be $g$. The FISTA method, as applied to $f_k$ and $g$, is expressed as
\begin{align}
    z^{k,j+1} & =p_{L,k}(y^{k,j}),            \label{eq:gen-fista-z}             \\
    y^{k,j+1} & =z^{k,j+1}+\tfrac{t_{j}-1}{t_{j+1}}(z^{k,j+1}-z^{k,j}), 
                                              \label{eq:gen-fista-y}
\end{align}
where $p_{L,k}(y)$ is defined as 
\begin{align} \label{eq:gen-fista-z2}
    p_{L,k}(y) &\doteq \argmin_{x\in\real^m} \left\{ g(x) + \frac{L}{2}\norm{x-\left(y-\frac{1}{L}\nabla f_k(y)\right)}^2  \right\} \\
    &= \prox_{(1/L)g}\!\big(y - (1/L)\nabla f_k(y)\big), \nonumber
\end{align}
and $t_{j+1}$ is a sequence of non-negative real numbers. Convergence the
function values $f_k(z^{k,j}) + g(z^{k,j})$ to the optimal value as
$j\to\infty$ is guaranteed if the sequence $t_{j+1}$ satisfies the following
inequality in~\cite[Theorem 4.1]{beck2009fast}:
\begin{equation} \label{eq:fista-t-inequality}
    (\forall\, j \geq 1)\quad t_{j+1}^2-t_{j+1} \leq t_{j}.
\end{equation}
The sequence given by~\cite{beck2009fast} in original FISTA is $t_{1}=1$ and
\begin{equation*}
    (\forall\, j \geq 1) \quad t_{j+1}=\frac{1+\sqrt{1+4t_{j}^2}}{2}.
\end{equation*}
Since our application requires convergence of iterates as well as function
values, we will instead employ the variant of FISTA introduced by Chambolle
and Dossal in~\cite{chambolle2015convergence}. This version of the algorithm
instead defines $t_j$ by
\begin{equation} \label{eq:fista-cd-tj}
    (\forall\, j\geq 1) \quad t_j=\frac{j+a-1}{a},
\end{equation}
where $a > 2$ is an arbitrary constant. It is easily verified that this
definition also satisfies~\eqref{eq:fista-t-inequality}; in addition, however,
it is shown in~\cite{chambolle2015convergence} that this form of $\{t_j\}$
also guarantees the convergence of (in the present notation)
$\{z^{k,j}\}_{j=1}^{\infty}$ to a minimizer of $f_k+g$. We denote this variant
of FISTA by ``FISTA-CD''.

The calculation of $z^{k,j+1} = p_{L,k}(y^{k,j}) =
\prox_{(1/L)g}\!\big(y^{k,j} - (1/L)\nabla f_k(y)\big)$ is essentially
identical to~\eqref{eq:abstract-subproblem-fb} with $\alpha_{k,j} = 1/L = 1/c_k$ and
$z^{k,j}$ replaced by $y^{k,j}$.  Therefore, following the same logic leading
from~\eqref{eq:abstract-subproblem-fb}
to~\eqref{eq:admm-like-x-nukj1}-\eqref{eq:admm-like-z-nukj1}, one obtains that an
equivalent calculation is 
\begin{align}
x^{k,j+1} &\in \argmin_{x\in\realn} 
           \left\{ f(x) + \inner{p^k}{Mx}  
             + \frac{c_k}{2}\norm{Mx - y^{k,j}}^2  \right\}, \label{eq:admm-xkj-cd} \\
z^{k,j+1} &= \argmin_{z\in\realm}\left\{g(z)-\left\langle p^k, z\right\rangle+\frac{c_k}{2}\left\| M x^{k, j+1}-z\right\|^2\right\}, \label{eq:admm-zkj-cd}
\end{align}
which is identical
to~\eqref{eq:admm-like-x-nukj1}-\eqref{eq:admm-like-z-nukj1}, except that
$z^{k,j}$ in~\eqref{eq:admm-like-x-nukj1} is replaced by $y^{k,j}$.

Expressing a subgradient of the augmented Lagrangian at $(x^{k,j},z^{k,j+1})$
in the manner of~\eqref{eq:subgradient-lagrange-kj} after this pair of operations leads to
\begin{align} \label{eq:subgradient-lagrange-kj0}
    s^{k,j+1} =
    \begin{pmatrix}
        s_x^{k,j+1} \\ s_z^{k,j+1}
    \end{pmatrix} \in
    \begin{pmatrix}
        \partial f(x^{k,j+1}) + M\transpose p^k + c_kM\transpose (Mx^{k,j+1}-z^{k,j+1}) \\
        \partial g(z^{k,j+1}) - p^k - c_k(Mx^{k,j+1}-z^{k,j+1})
    \end{pmatrix}
    =
    \begin{pmatrix}
    s_x^{k,j+1} \\ 0
    \end{pmatrix},
\end{align}
where the $s_z^{k,j+1}$ component being zero is an immediate consequence of
the natural necessary optimality condition for~\eqref{eq:admm-zkj-cd}.
Rewriting the expression for the $x$ component of the subgradient yields
\begin{align}
   &\quad\;\partial f(x^{k,j+1}) + M\transpose p^k + c_k M\transpose(Mx^{k,j+1}-z^{k,j+1})                               \nonumber  \\
   & = \partial f(x^{k,j+1}) + M\transpose p^k + c_k M\transpose(Mx^{k,j+1}-y^{k,j}) + c_k M\transpose (y^{k,j}-z^{k,j+1}) \nonumber \\
   & \ni 0 + c_k M\transpose (y^{k,j}-z^{k,j+1}) \nonumber \\
   & = c_k M\transpose(y^{k,j}-z^{k,j+1}), \label{simplifiedSubgrad}
\end{align}
where the ``$\ni$'' relation holds from the standard necessary optimality
condition $0 \in \partial f(x^{k,j+1}) + M\transpose p^k + c_k
M\transpose(Mx^{k,j+1} - y^{k,j})$ for~\eqref{eq:admm-xkj-cd}.  So, we may
select the subgradients $s_x^{k,j+1} = c_k M\transpose(y^{k,j}-z^{k,j+1})$ and
$s_z^{k,j+1} = 0$.  Having $s_z^{k,j+1} = 0$ causes the acceptance
criterion~\eqref{eq:inner-stopping} to simplify to
\begin{equation} \label{eq:fista-inner-stopping}
    \frac{2\rho_k}{c_k}
    \abs{\inner{x^{k,j+1} - w_x^k}{s_x^{k,j+1}}} + \rho_k^2 \norm{s^{k,j+1}}^2
     \leq \parent{2\rho_k - \rho_k^2 - \epsilon} \norm{Mx^{k,j+1}-z^{k,j+1}}^2. 
\end{equation}
Assuming that we will use the adaptive relaxation procedure proposed in
Section~\ref{sec:ways-to-use}, the discriminant in~\eqref{eq:defDeltak} may be
written
\begin{align} 
\Delta_{k,j} 
&\doteq \left(\norm{u^{k,j+1}}^2 
      - \frac{1}{c_k} \abs{\inner{x^{k,j+1} - w_x^k}{s_x^{k,j+1}}} 
      \right)^2
         - \epsilon\left( \norm{u^{k,j+1}}^4 
          + \norm{u^{k,j+1}}^2 \norm{s_x^{k,j+1}}^2 \right) \nonumber 
         \\
&= \left(\norm{u^{k,j+1}}^2 
       - \frac{1}{c_k} \abs{\inner{x^{k,j+1} - w_x^k}{c_k M\transpose(y^{k,j}-z^{k,j+1})}} 
       \right)^2 \nonumber \\
&\qquad \qquad \qquad \qquad \qquad \qquad \qquad \qquad
          - \; \epsilon \left( \norm{u^{k,j+1}}^4 
           + \norm{u^{k,j+1}}^2 \norm{s_x^{k,j+1}}^2 \right) \nonumber \\
&= \left(\norm{u^{k,j+1}}^2 - \abs{(y^{k,j}-z^{k,j+1})\transpose M (x^{k,j+1} - w_x^k)} \right)^2
           \nonumber \\
&\qquad \qquad \qquad \qquad \qquad \qquad \qquad \qquad
          - \; \epsilon \left( \norm{u^{k,j+1}}^4 
           + \norm{u^{k,j+1}}^2 \norm{s_x^{k,j+1}}^2 \right), \label{eq:specific-Delta-kj}
\end{align}
with $u^{k,j+1} = z^{k,j+1} - Mx^{k,j+1}$, as dictated by~\eqref{eq:sum-F}.
Specializing the definitions of $U_k$, $A_k$, and $S_k$ given above to the
current choice of $F$, we have
\begin{align*}
U_{k,j} &\doteq \norm{u^{k,j+1}}^2 = \norm{Mx^{k,j+1} - z^{k,j+1}}^2 \\
A_{k,j} &\doteq \abs{(y^{k,j}-z^{k,j+1})\transpose M (x^{k,j+1} - w_x^k)} \\
S_{k,j} &\doteq \norm{s_x^{k,j}}^2 = c_k^2 \norm{M\transpose(y^{k,j}-z^{k,j+1})}^2,
\end{align*}
in which case~\eqref{eq:specific-Delta-kj} simplifies to
\begin{equation}
\Delta_{k,j} = (U_{k,j} - A_{k,j})^2 - \epsilon (U_{k,j}^2 + U_{k,j} S_{k,j}),
\end{equation}
and, when $\Delta_{k,j} \geq 0$, the allowable range for the relaxation factor
$\rho_k$ as specified in~\eqref{eq:rhok-range-ineq} simplifies to 
\begin{equation*}
\frac{U_{k,j} - A_{k,j} - \sqrt{\Delta_{k,j}}}{U_{k,j} + S_{k,j}}
\leq \rho_k \leq
\frac{U_{k,j} - A_{k,j} + \sqrt{\Delta_{k,j}}}{U_{k,j} + S_{k,j}},
\end{equation*}
which may be modeled by selecting any $\gamma_k \in [-1,1]$ and setting
\begin{equation}
\rho_k = \frac{U_{k,j} - A_{k,j} + \gamma_k\sqrt{\Delta_{k,j}}}{U_{k,j} + S_{k,j}}.
\end{equation}
Using these results and including the calculations~\eqref{eq:gen-fista-y}
and~\eqref{eq:fista-cd-tj} from the FISTA-CD method leads to
Algorithm~\ref{alg:ALM-AR-FISTA-CD}, whose behavior is analyzed in the
next section.

\begin{algorithm}[t]{}
    \caption{ALM-AR-FISTA-CD Algorithm} \label{alg:ALM-AR-FISTA-CD}
    \begin{algorithmic}[1]
        \State \textbf{Initialization} Choose $p^0 \in \real^m$, $z^{0} \in \real^m$, $w^{0}=0$, and pick $a > 2$, $\epsilon \in (0,1)$
        \For {$k=0,1,2, \cdots$} \algorithmiccomment{Outer loop}
        \State Choose $c_k>0$, set $z^{k,1}=y^{k,1}=z^{k}$, $t_{k,1}=1$ 
               \label{step:innerLoopInit}
        \Repeat{~for $j=1, 2, \ldots$} \algorithmiccomment{Inner loop}
        \State {$x^{k,j+1} \in \argmin_{x\in\real^n} \left\{f(x) + \inner{p^k}{Mx}  + \frac{c_k}{2}\norm{Mx - y^{k,j}}^2  \right\}$} \label{step:minx}
        \State {$z^{k,j+1}=\argmin_{z\in \real^m} \left\{ g(z) -\inner{p^k}{z} + \frac{c_k}{2} \norm{Mx^{k,j+1} - z}^2 \right\}$} \label{step:minz}
        \State {$t_{k, j+1}=\tfrac{j+a-1}{a}$} \label{step:nextt} 
        \State {$y^{k, j+1}=z^{k,j+1}+\frac{t_{k,j}-1}{t_{k,j+1}}(z^{k,j+1}-z^{k,j})$}
               \label{step:calcy}
        \State {$s^{k,j+1} = c_k M\transpose(y^{k,j}-z^{k,j+1})$}
               \label{step:calcsjk}
        \State {$U_{k,j} = \norm{Mx^{k,j+1} -  z^{k,j+1}}^2$}
        \State {$S_{k,j} = \norm{s^{k,j+1}}^2$}
        \State {$A_{k,j} = \abs{(y^{k,j}-z^{k,j+1})\transpose M (x^{k,j+1} - w^k)}$}
        \State {$\Delta_{k,j}= (U_{k,j} - A_{k,j})^2 - \epsilon (U_{k,j}^2 + U_{k,j} S_{k,j})$}
        \Until{$A_{k,j}<U_{k,j}$ and $\Delta_{k,j} \geq 0$} \label{step:AAFCendInner}
        \State $x^{k+1}=x^{k,j+1} \quad z^{k+1}=z^{k,j+1} \quad s^{k+1}=s^{k,j+1}$
        \State $U_k = U_{k,j}$ \quad $S_k = S_{k,j}$ \quad $A_k = A_{k,j}$ 
                               \quad $\Delta_k = \Delta_{k,j}$ 
        \State $\gamma_k^{\min} = \max\left\{-1, \frac{A_k - U_k}{\sqrt{\Delta_k}}
                                           \right\}$
        \State Choose any $\gamma_k \in (\gamma_k^{\min},1]$
        \State $\rho_k = \frac{U_k - A_k + \gamma_k\sqrt{\Delta_k}}{U_k + S_k}$
        \State $w^{k+1} = w^{k} - \rho_k c_k s^{k+1}$
        \State $\bar p^k = p^k + c_k(Mx^{k+1} - z^{k+1})$
        \State $p^{k+1} = p^k + \rho_k c_k(Mx^{k+1}-z^{k+1})$
        \EndFor
    \end{algorithmic}
\end{algorithm}

\subsection{Convergence}
The convergence of the FISTA-Type algorithm relies strongly on the convergence
of the iterates of FISTA-CD as shown in~\cite{chambolle2015convergence} (in a
general Hilbert space, although here we only require convergence in the
special case of $\real^n$).  For completeness, we state their theorem here,
converted to the present application and notation:
\begin{theorem}
~\cite[Theorem 4.1]{chambolle2015convergence} Let $a>2$ be a positive real
number, and for all $j \in \mathbb{N}$ let $t_j=\frac{j+a-1}{a}$. 
If there exists at least a solution to the subproblem~\eqref{eq:defbal-prob}, 
then the sequence $\{z^{k,j}\}_{j=1}^{\infty}$ generated by the FISTA-CD 
algorithm converges to a minimizer of $\phi_k(z)=f_k(z) + g(z)$.
\label{thm:fista-cd-convergence}
\end{theorem}
Using this result, it is now possible to prove the convergence of 
ALM-AR-FISTA-CD algorithm:

\begin{theorem}\label{thm:ALM-AR-FISTA-CD-convergence}
The ALM-AR-FISTA-CD algorithm described in Algorithm~\ref{alg:ALM-AR-FISTA-CD}
has the following behavior:
\begin{enumerate}[label=(\roman*), nosep]
    \item If every instance of the inner loop (over $j$) terminates finitely
    and a saddle point of $L$ defined in~\eqref{eq:conjugate-lagrange} exists,
    then the algorithm's outer loop (over $k$) produces sequences
    $\{x^k\}$, $\{z^k\}$, and $\{p^k\}$ such that $\{p^k\}$ converges to a
    solution of the dual problem $\min_{\in\real^m} f^*(-M\transpose p) +
    g^*(p)$, while $F\big( (x^k, z^k), z^k-Mx^k \big)$ converges to the
    optimal value of the primal problem~\eqref{eq:conjugate-primal} and any
    limit points $(x^\infty,z^\infty)$ (if they exist) of the $\big\{(x^k,
    z^k)\big\}$ sequence are such that $x^\infty$ solves the primal
    problem~\eqref{eq:conjugate-primal} and $z^\infty = Mx^\infty$.

    \item If every instance of the inner loop terminated finitely and a saddle
    point of $L$ does not exist, then at least one of the sequences $\{x^k\}$,
    $\{z^k\}$, or $\{p^k\}$ generated by the outer loop is unbounded.

    \item If for some value $k$ of the outer loop index, the inner loop runs
    indefinitely, then $p^k$ is an optimal solution of the dual problem and
    the sequences $\{x^{k,j}\}_{j=0}^\infty$ and $\{z^{k,j}\}_{j=0}^\infty$
    generated by the inner loop are asymptotically optimal in the sense that
    \begin{align*}
    \limsup_{j\rightarrow\infty} f(x^{k,j})+g(z^{k,j}) 
         &\leq \inf_{x\in\real^n}\big\{f(x)+g(Mx)\big\} &
    \lim_{j\rightarrow\infty} \norm{Mx^{k,j}-z^{k,j}} &= 0.
    \end{align*}
\end{enumerate}
\end{theorem}

\begin{proof}
See~\ref{app:ALM-AR-FISTA-CD-convergence-proof}. 
\end{proof}

\section{Numerical Experiments with LASSO Regression}
\label{sec:numerical-experiments}

A simple problem class that fits the form~\eqref{eq:primal-problem} with $g$
having an easily computed proximal map is LASSO regression problem, which may
be written as
\begin{equation}
    \min_{x\in\mathbb{R}^n} \tfrac{1}{2}\norm{Ax-b}^2 + \nu\norm{x}_1,
\end{equation}
where $A$ is a $p \times n$ matrix, $b\in \real^p$ and $\nu>0$ is a
regularization parameter. This problem readily fits the
form~\eqref{eq:primal-problem} with $f(x) = \tfrac{1}{2}\norm{Ax-b}^2$, $g(x)
= \nu\norm{x}_1$, and $M = I$. The $f$ subproblem reduces to solving a system
of linear equations, and the $g$ subproblem reduces to the soft-thresholding
operator.  For each problem instance we tested, we scaled $b$ and the columns
of $A$ to have unit $\ell_2$ norm, and set the $\nu$ regularization parameter
to $0.1\norm{A\transpose b}_\infty$.

Our experiments compare five algorithms:
\begin{description}
  \item[ALM-ADSS] The augmented Lagrangian outer loop fixes $\rho_k \equiv 1$,
  and the inner loop is alternating minimization, shown in
  Section~\ref{sec:connection-between-admm} to be equivalent to applying the
  proximal gradient method to a dual formulation of the subproblem, using a
  constant stepsize of $1/c_k$.  The acceptance criterion to end the inner
  loop is~\eqref{eq:sum-frame-approx} with $\rho_k$ fixed to $1$, which is equivalent to the acceptance condition in~\cite{eckstein2013practical}.
  \item[ALM-AR-ADSS] The augmented Lagrangian outer loop uses adaptive
  relaxation, as Algorithm~\ref{alg:gen-relax}, while the inner loop is
  alternating relaxation.
  \item[ALM-FISTA-CD] The augmented Lagrangian outer loop fixes $\rho_k \equiv
  1$, while the inner loop is the FISTA-CD method, essentially
  steps~\ref{step:minx}-\ref{step:calcsjk} of
  Algorithm~\ref{alg:ALM-AR-FISTA-CD}.  Again, the acceptance criterion to end
  the inner loop is~\eqref{eq:sum-frame-approx} with $\rho_k$ fixed to $1$.
  \item[ALM-AR-FISTA-CD] Algorithm~\ref{alg:ALM-AR-FISTA-CD}, which uses both adaptive relaxation and the FISTA-CD subproblem solver.
  \item[ADMM] The traditional ADMM.
\end{description}

The termination condition for all the algorithms was
\begin{equation*}
    \dist_\infty\!\big(0, \partial[f(x) + g(Mx)]_{x=x^k}\big) \leq \delta, 
\end{equation*}
where $\dist_\infty(t, S)=\inf\big\{\norm{t-s}_\infty \big|\; s\in S \big\}$
and $\delta$ is a tolerance parameter set to be $10^{-6}$. 

We tested three categories of LASSO datasets:
\begin{description}
\item[Single-Pixel Camera:] (``Pixel'') Four dense compressed image sensing datasets 
from~\cite{duarte2008single}, with $m\in[410, 4770]$ and $n\in[1024, 16384]$. 
These problems are referred as \textit{Mug32}, \textit{Ball64}, \textit{Logo64} 
and \textit{Mug128}.

\item[Gene expression:] (``Gene'') Six standard cancer DNA microarray datasets 
from~\cite{dettling2004finding}. These instances have dense, wide and relatively
small matrices $A$, with $m\in[42, 102]$ and $n\in[2000, 6033]$. These problems
are referred as \textit{Brain}, \textit{Colon}, \textit{Leukemia}, 
\textit{Lymphoma}, \textit{Prostate} and \textit{SRBCT}.

\item[Jet Engine Reliability:] (``Engine'') Four multivariate jet engine
reliability time series datasets from~\cite{saxena2008damage}, with
$m\in[20631, 61249]$ and $n=24$. These problems are referred as
\textit{Fd001}, \textit{Fd002}, \textit{Fd003} and \textit{Fd004}.
\end{description}
Table~\ref{tab:lasso-params} shows the dimensions of the test instances.

\begin{table}[h!]
    \caption{LASSO regression test problem parameters.}
    \label{tab:lasso-params}
    \centering
    \begin{tabular}{|l|l|r|r|}
        \hline
        \textbf{Category} &
        \textbf{Instance} & \textbf{Observations} & \textbf{Features}
        \\ \hline \hline
        Pixel 
        & Mug32   & 410          & 1,024     \\ \hline
        & Ball64  & 1,638        & 4,096     \\ \hline
        & Logo64  & 1,638        & 4,096     \\ \hline
        & Mug128  & 4,770        & 16,384    \\ \hline \hline
        Gene 
        & Brain   & 42           & 5,597     \\ \hline
        & Colon   & 62           & 2,000     \\ \hline
        & Leukemia& 72           & 3,571     \\ \hline
        & Lymphoma& 62           & 4,026     \\ \hline
        & Prostate& 102          & 6,033     \\ \hline
        & SRBCT   & 63           & 2,308     \\ \hline \hline
        Engine 
        & Fd001   & 20,631       & 24        \\ \hline
        & Fd002   & 53,759       & 24        \\ \hline
        & Fd003   & 24,720       & 24        \\ \hline
        & Fd004   & 61,249       & 24        \\ \hline \hline
    \end{tabular}
    \vspace{2ex}
\end{table}

After a modest amount of experimentation, we set the algorithm parameters as follows, attempting to find good values for each individual algorithm:
\begin{itemize}[nosep]
  \item $\epsilon=0.1$
  \item $\gamma_k$ is always chosen as large as possible, so that $\rho_k$ is always as large as possible
  \item For the FISTA-CD inner loop, $a=3$
  \item $c_k$ was optimized by algorithm and test instance category (but not
  by individual instance), as shown in Table~\ref{tab:lasso-stepsize}.
  \item For the first $J_1$ iterations of the inner loop of the adaptive
  relaxation algorithms, we impose a stronger termination condition that it
  should be possible to set $\rho_k \geq 1$.  After some algebraic
  manipulations, this change is equivalent to strengthening the condition
  $\Delta_{k,j} \geq 0$ in step~\ref{step:AAFCendInner} of
  Algorithm~\ref{alg:ALM-AR-FISTA-CD} to $\Delta_{k,j} \geq (A_{k,j} +
  S_{k,j})^2$.  Once $j > J_1$, this condition is weakened to $\Delta_{k,j}
  \geq 0$ and underrelaxed steps are accepted.  The rationale for this
  heuristic is that for subproblems that appear relatively easy to solve, it
  may be worth prolonging the inner loop in the hope of finding a relatively
  accurate subproblems solution which will in turn lead to a relatively large
  multiplier adjustment.  The parameter $J_1$ was set by problem category as
  shown in Table~\ref{tab:lasso-lower-bound-threshold}.
  \item In the adaptive relaxation algorithms, when the number of inner iterations in the  outer iteration exceeded the threshold $J_{\mathrm{r}}$ given in 
  Table~\ref{tab:lasso-reset-threshold}, we reset $w^{k+1}$ to $x^{k+1}$.  
  This heuristic prevents the $\{x^k\}$ and $\{w^k\}$ 
  sequences from differing excessively from one another, a phenomenon that 
  can cause the termination condition for the inner
  loop to become too stringent in practice.
\end{itemize}

\begin{table}[h!]
\caption{Stepsize $c_k$ setup}
\label{tab:lasso-stepsize}
\centering
\begin{tabular}{|l|r|r|r|r|r|}
    \hline
    \textbf{Instance} & & \multicolumn{4}{|c|}{\textbf{ALM-}} \\
    \textbf{Category} & \textbf{ADMM} & \textbf{FISTA-CD} 
    & \textbf{AR-FISTA-CD} & \textbf{ADSS} & \textbf{AR-ADSS} \\
    \hline
    Pixel  & 2    & 3        & 3           & 2      & 2       \\ \hline
    Cancer & 2    & 4        & 4           & 3      & 7       \\ \hline
    Engine & 0.01 & 0.007    & 0.009       & 0.0007 & 0.0006  \\ \hline
\end{tabular}
\vspace{2ex}
\end{table}

\begin{table}[h!]
\caption{Setup of the overrelaxation threshold parameter $J_1$}
\label{tab:lasso-lower-bound-threshold}
\centering
\begin{tabular}{|l|r|r|}
    \hline
    \textbf{Instance} & \multicolumn{2}{|c|}{\textbf{ALM-}} \\
    \textbf{Category} & \textbf{AR-FISTA-CD} & \textbf{AR-ADSS} \\
    \hline
    Pixel  & 2      & 1       \\ \hline
    Cancer & 6      & 1       \\ \hline
    Engine & 6      & 1       \\ \hline
\end{tabular}
\vspace{2ex}
\end{table}

\begin{table}[h!]
\caption{Reset threshold $J_{\mathrm{r}}$ setup}
\label{tab:lasso-reset-threshold}
\centering
\begin{tabular}{|l|r|r|r|r|}
    \hline
    \textbf{Instance} & \multicolumn{4}{|c|}{\textbf{ALM-}} \\
    \textbf{Category} & \textbf{FISTA-CD} & \textbf{AR-FISTA-CD} & \textbf{ADSS} & \textbf{AR-ADSS} \\
    \hline
    Pixel  & 3        & 4           & 4      & 1       \\ \hline
    Cancer & 3        & 2           & 10     & 1       \\ \hline
    Engine & 10       & 7           & 10     & 1       \\ \hline
\end{tabular}
\vspace{2ex}
\end{table}

The initial values of $p^0$, $w^0$, and $z^0$ were zero. All the implementations
were in Python, using the NumPy package.

\begin{table}[p]
\caption{Outer iterations for LASSO problems.}
\label{tab:numer-lasso-outer}
\centering
\begin{tabular}{|l|r|r|r|r|r|}
    \hline
    & & \multicolumn{4}{|c|}{\textbf{ALM-}} \\
    \textbf{Instance} & \textbf{ADMM} & \textbf{FISTA-CD} 
        & \textbf{AR-FISTA-CD} & \textbf{ADSS} & \textbf{AR-ADSS} 
        \\ \hline \hline
    Mug32    & 314    & 15     & 25     & 7       & 28      \\ \hline
    Ball64   & 259    & 90     & 51     & 29      & 81      \\ \hline
    Logo64   & 264    & 87     & 55     & 29      & 82      \\ \hline
    Mug128   & 551    & 314    & 175    & 96      & 250     \\ \hline
    Colon    & 665    & 158    & 136    & 67      & 152     \\ \hline
    Prostate & 1,344  & 690    & 551    & 275     & 632     \\ \hline
    Brain    & 2,045  & 1,050  & 828    & 415     & 833     \\ \hline
    Leukemia & 612    & 233    & 205    & 104     & 261     \\ \hline
    Lymphoma & 822    & 343    & 280    & 151     & 405     \\ \hline
    SRBCT    & 615    & 113    & 104    & 45      & 232     \\ \hline
    Fd001    & 1,116  & 272    & 154    & 3,368   & 2,062   \\ \hline
    Fd002    & 1,056  & 958    & 549    & 13,219  & 8,819   \\ \hline
    Fd003    & 945    & 334    & 245    & 3,338   & 2,087   \\ \hline
    Fd004    & 1,188  & 701    & 362    & 6,313   & 3,828   \\ \hline \hline
    \textbf{Geometric Mean} & 713.44 & 241.73 & 178.78 & 245.36 & 448.33  \\ \hline
\end{tabular}
\end{table}

Table~\ref{tab:numer-lasso-outer} show the number of outer iterations (number
of multiplier updates) required to solve the LASSO regression instances. The
ALM methods generally solve problems with significantly fewer outer iterations
than that ADMM, in accord with earlier results such as those in~\cite{PJO}.

\begin{table}[p]
\caption{Cumulative inner iterations for LASSO problems.}
\label{tab:numer-lasso-inner}
\centering
\begin{tabular}{|l|r|r|r|r|r|}
    \hline
    & & \multicolumn{4}{|c|}{\textbf{ALM-}} \\
    \textbf{Instance} & \textbf{ADMM} & \textbf{FISTA-CD} & 
        \textbf{AR-FISTA-CD} & \textbf{ADSS} & \textbf{AR-ADSS} 
        \\ \hline \hline
    Mug32    & 314    & 298    & 181    & 544     & 374     \\ \hline
    Ball64   & 259    & 196    & 108    & 398     & 297     \\ \hline
    Logo64   & 264    & 159    & 127    & 355     & 350     \\ \hline
    Mug128   & 551    & 421    & 241    & 666     & 368     \\ \hline
    Colon    & 665    & 712    & 531    & 3,133   & 2,439   \\ \hline
    Prostate & 1,344  & 1,544  & 1,401  & 6,894   & 4,507   \\ \hline
    Brain    & 2,045  & 2,324  & 1,919  & 9,659   & 5,859   \\ \hline
    Leukemia & 612    & 1,047  & 929    & 2,375   & 2,213   \\ \hline
    Lymphoma & 822    & 1,289  & 893    & 3,757   & 3,582   \\ \hline
    SRBCT    & 615    & 625    & 612    & 6,790   & 6,267   \\ \hline
    Fd001    & 1,116  & 602    & 348    & 3,794   & 2,515   \\ \hline
    Fd002    & 1,056  & 1,859  & 750    & 20,573  & 8,958   \\ \hline
    Fd003    & 945    & 443    & 671    & 3,585   & 2,360   \\ \hline
    Fd004    & 1,188  & 1,195  & 978    & 9,059   & 4,007   \\ \hline \hline
    \textbf{Geometric Mean} & 713.44 & 684.84 & 507.06 & 2,779.20 & 1,927.24  \\ \hline
\end{tabular}
\end{table}

Table~\ref{tab:numer-lasso-inner} shows the cumulative number of inner
iterations (number of $x$ or $z$ updates) required to solve the LASSO
regression problems  (the cumulative number of inner iterations is closely
correlated with run time for all the algorithms).  Using FISTA-CD as the inner
loop significantly improve the performance of the ALM-ADSS algorithm. Adaptive
relaxation also generally improves performance. Used together, these two
techniques improve ALM-ADSS to a performance close to or even better than that
of the ADMM across three different datasets. Across the datasets tested,
combining the FISTA-CD inner loop technique with adaptive relaxation results
in a 29\% reduction in the geometric mean of the number of total inner
iterations in comparison to the ADMM, albeit at the cost of a significantly
more elaborate algorithm.

\bibliographystyle{abbrv}
\bibliography{references}

\appendix
\renewcommand{\thesection}{Appendix \Alph{section}}
\section{Proof of Proposition~\ref{prop:unbounded}}\label{app:unbounded-proof}
\begin{proof}
Since $L$ is maximal monotone, then by~\cite{Minty61} or~\cite[Theorem
12.41]{rockafellar1970convex} its domain $\dom L$ is a nonempty nearly convex
set, meaning that there exists a convex set $X$ such that $\dom L
\subseteq X \subseteq \closure \dom L$.  Using~\cite[Lemma 2.7]{BMW13}, one
then has $\ri \dom L = \ri X$, which is necessarily nonempty since $X$ is
nonempty convex and hence must have a nonempty relative interior.  Let
$(\tilde x, \tilde p)$ be an arbitrary element of the nonempty set $\ri \dom
L$.

As in~\cite{eckstein2013practical}, the analysis will use a variant of the
analysis originally presented in~\cite{Roc76a} for the behavior of the
proximal point algorithm for operators with no roots. This analysis proceeds
by contradiction. Suppose that no saddle point exists, but that both
$\{x^k\}$ and $\{\bar p^k\}$ are bounded. In this case, there must exist
some positive scalar $R\in\real$ such that $\sup_{k\geq 0}\{
\smallnorm{(x^{k+1}, \bar p^k)} \} < R$.  Furthermore, take $R$ large enough
that $R > \smallnorm{(\tilde x, \tilde p)}$. Let $B_R = \{ (x, p) \in
\realn\times\realm:\smallnorm{(x, p)} \leq R \}$ be the closed ball of radius
$R$ centered at the origin, and consider the point-to-set mapping $T =
\partial L + N_{B_R}$, where $N_{B_R}$ is the normal cone of $B_R$.

Now, since $R > \smallnorm{(\tilde x, \tilde p)}$, we have 
\begin{equation*}
(\tilde x, \tilde p) \in \ri \dom N_{B_R} 
   = \{ (x, p) \in \realn\times\realm:\smallnorm{(x, p)} < R \}.
\end{equation*}
Since $(\tilde x, \tilde p) \in \ri \dom L$, it follows that $(\tilde x,
\tilde p) \in \ri \dom L \cap \ri \dom N_{B_R}$, establishing that $\ri \dom L
\cap \ri \dom N_{B_R} \neq \emptyset$. Then, using~\cite[Theorem
12.44]{rockafellar1970convex}, it follows that $T =
\partial L + N_{B_R}$ is maximal monotone. Furthermore, since $\dom T$ lies within
the bounded set $B_R$, it follows from~\cite[Proposition
2]{rockafellar1969local} that there must exist at least one point $(x^*,
p^*)\in \realm \times \realm$ such that $(0, 0) \in T(x^*, p^*)$. Since the
sequence $\brac{(x^{k+1}, \bar p^k)}$ lies within the interior of $B_R$, it
follows that $N_{B_R}(x^{k+1}, \bar p^k) = \{0\}$ for all $k\geq 0$. Then we
have for all $k\geq 0$ that
\begin{equation*}
    T(x^{k+1}, \bar p^k) = \partial L(x^{k+1}, \bar p^k) + N_{B_R}(x^{k+1}, \bar p^k) = \partial L(x^{k+1}, \bar p^k) \Rightarrow (s^{k+1}, u^{k+1}) \in T(x^{k+1}, \bar p^k).
\end{equation*}
Using the monotonicity of $T$, one has for all $k\geq 0$ that 
\begin{equation*}
    \inner{x^{k+1} - x^*}{s^{k+1}-0} + \inner{\bar{p}^k - p^*}{u^{k+1}-0} \geq 0,
\end{equation*}
which takes the same form as~\eqref{eq:monotonicity}, the only
difference being that $(x^*,p^*)$ is taken to be a root of $T$ instead of $L$.
Combining the above inequality with~\eqref{eq:shrink-pw-eq-adaptive}, and
following the same logic as the case in which a saddle points exist, we still
obtains all the results above from~\eqref{eq:shrink-pw-almost-adaptive}
through~\eqref{eq:decompose-inner}, but with the slightly different assumption
regarding $(x^*,p^*)$.  The deductions that $c_k u^{k+1} \to 0$, $u^{k+1} \to 0$,
$\inner{s^{k+1}}{x^{k+1}} \to 0$, and $s^k \to 0$ from this sequence of
deductions remain valid.

From the boundedness of $\brac{(x^{k+1}, \bar p^k)}$, it must have at least
one limit point $(x^{\infty}, p^{\infty})$, which by the construction of $R$
must lie within the interior of $B_R$. Taking the limit of $(s^{k+1}, u^{k+1})
\in T(x^{k+1}, \bar p^k)$ over some converging subsequence $\mathcal{K}$ and
using the maximality of $T$, we have $(0,0) \in T(x^{\infty}, p^{\infty})$.
Since $(x^{\infty}, p^{\infty})$ lies in the interior of $B_R$, we have
$N_{B_R}(x^{\infty}, p^{\infty}) = \{0\}$ and so
$L(x^{\infty}, p^{\infty}) = T(x^{\infty}, p^{\infty})$, which yields $(0, 0)
\in \partial L(x^{\infty}, p^{\infty})$, meaning $(x^{\infty}, p^{\infty})$ is
a saddle point of $L$. This contradicts the assumption that no saddle point
exists. Therefore, in the case that no saddle point exists, at least one of
the sequences $\{x^k\}$ or $\{\bar p^k\}$ must be unbounded.
\end{proof}

\section{Proof of Lemma~\ref{lem:genericInnerLoop}}
\label{app:inner-loop-proof}
\begin{proof}
Let $k$ be as specified, with the corresponding inner loop over $j$ never
terminating. Since $\lim_{j\to\infty} s^{k,j} = 0$ from the compatible
subproblem process definition, it follows that $\lim_{j\to\infty} S_{k,j} =
\lim_{j\to\infty} \smallnorm{s^{k,j}}^2 = 0$. Furthermore, 
$\big\{(x^{k,j},u^{k,j},s^{k,j})\big\}_{j=0}^\infty$ being
a compatible subproblem process also means that $\lim_{j\to\infty}
\inner{s^{k,j}}{x^{k,j}} = 0$, so by the calculation of $A_{k,j}$ in
step~\ref{step:calcAkj} of the algorithm,
\begin{align*}
\lim_{j\to\infty} A_{k,j}
&= \lim_{j\to\infty} \left\{ \frac{1}{c_k}\abs{\inner{x^{k,j}-w^k}{s^{k,j}}} \right\} \\
&= \frac{1}{c_k} \lim_{j\to\infty}  \abs{\inner{x^{k,j}}{s^{k,j}}
                                                 - \inner{w^k}{s^{k,j}}} \\
&= \frac{1}{c_k} \abs{\lim_{j\to\infty} \left\{ \inner{x^{k,j}}{s^{k,j}}
                                                 - \inner{w^k}{s^{k,j}} \right\}} \\
&= \frac{1}{c_k} \abs{\lim_{j\to\infty} \inner{x^{k,j}}{s^{k,j}}
                                            - \lim_{j\to\infty} \inner{w^k}{s^{k,j}} } \\
&= \frac{1}{c_k}\abs{0 - 0} = 0.
\end{align*}
It remains only to show that $\lim_{j\to\infty} U_{k,j} = 0$.
To that end, define
\begin{align*}
\mathcal{J}_1 &\doteq \set{j\in\nat}{A_{k,j} \geq U_{k,j}} &
\mathcal{J}_2 &\doteq \set{j\in\nat}{\Delta_{k,j} < 0}.
\end{align*}
Having the inner loop for outer step $k$ of the algorithm fail to terminate is
equivalent to $\mathcal{J}_1 \cup \mathcal{J}_2 = \nat$,
that is, every $j\in\nat$ being a member of at least one of $\mathcal{J}_1$,
$\mathcal{J}_2$, or both.  Thus, at least one of these index
sets must be infinitely large.  

Suppose that $\card{\mathcal{J}_1} = \infty$.  Then, considering that
$A_{k,j} \geq U_{k,j}$ for all $j\in\mathcal{J}_1$, that $U_{k,j}$ cannot be negative for any $j$, and that $\lim_{j\to\infty} A_{k,j} = 0$, it follows that
$\lim_{j\to\infty,j\in\mathcal{J}_1} U_{k,j} = 0$.

We next claim that if
$\card{\mathcal{J}_2} = \infty$, then $\lim_{j\to\infty,j\in\mathcal{J}_2}
U_{k,j} = 0$.  This is the most involved element of the proof and proceeds by
contradiction.  Suppose in this case that it is not true that
$\lim_{j\to\infty,j\in\mathcal{J}_2} U_{k,j} = 0$.  Since $U_{k,j}$ cannot be
negative for any $j$ by definition, it then follows that there must exist some
$\delta > 0$ and infinitely large $\mathcal{J}_2' \subseteq \mathcal{J}_2$
such that $U_{k,j} > \delta$ for all $j\in\mathcal{J}_2'$.  For all $j \geq
0$, also observe that
\begin{align}
\Delta_{k,j} &= (U_{k,j} - A_{k,j})^2 - \epsilon U_{k,j} (U_{k,j} + S_{k,j}) 
 \nonumber \\
&= U_{k,j}^2 - 2 A_{k,j} U_{k,j} + A_{k,j}^2 
          - \epsilon U_{k,j}^2 - \epsilon S_{k,j} U_{k,j} \nonumber \\
&= (1-\epsilon) U_{k,j}^2 - (2 A_{k,j} + \epsilon S_{k,j}) U_{k,j} + A_{k,j}^2 
 \nonumber \\
&= U_{k,j} \big( (1-\epsilon)U_{k,j} - 2A_{k,j} - \epsilon S_{k,j}\big) + A_{k,j}^2 
 \nonumber \\
&\geq U_{k,j} \big( (1-\epsilon)U_{k,j} - (2A_{k,j} + \epsilon S_{k,j}) \big), 
  \label{UkjLowerBound}
\end{align}
Now, since $\lim_{k\to\infty} A_{k,j} = 0$ and $\lim_{k\to\infty} S_{k,j} =
0$, it follows that $\lim_{k\to\infty} \{ 2A_{k,j} + \epsilon S_{k,j} \} = 0$.
So, there must exist some $j_0 \in \nat$ such that $2A_{k,j} + \epsilon
S_{k,j} < (1-\epsilon)\delta/2$ for all $j \geq j_0$. Since $U_{k,j} >
\delta$ for all $j\in\mathcal{J}_2'$, it then follows that
\begin{align*}
(\forall\, j \in \mathcal{J}_2' : j \geq j_0) \qquad
(1-\epsilon) U_{k,j} - (2A_{k,j} + \epsilon S_{k,j})
&> (1-\epsilon) \delta - \frac{(1-\epsilon)\delta}{2}
= \frac{(1-\epsilon)\delta}{2} > 0.
\end{align*}
Substituting this inequality into~\eqref{UkjLowerBound} and using once again
that $U_{k,j} > \delta$ for all $j\in\mathcal{J}_2'$, it then follows that
\begin{align*}
(\forall\, j \in \mathcal{J}_2' : j \geq j_0) \qquad
\Delta_{k,j} &> U_{k,j} \frac{(1-\epsilon)\delta}{2} 
\geq \frac{(1-\epsilon)\delta^2}{2}.
\end{align*}
Since $\mathcal{J}_2' \subseteq \mathcal{J}_2$, one may then conclude that
$\liminf_{j\to\infty,j\in\mathcal{J}_2} \Delta_{k,j} 
\geq (1-\epsilon)\delta/2 > 0$.  However, this result is impossible because,
by definition, $\Delta_{k,j} < 0$ for every $j\in\mathcal{J}_2$.  Therefore it
must be the case that $\lim_{j\to\infty,j\in\mathcal{J}_2} U_{k,j} = 0$.

Finally, let $\mathcal{L}$ be a set of sequence of indices such
that $\lim_{j\to\infty,j\in\mathcal{L}} U_{k,j} = \limsup_{j\to\infty}
U_{k,j}$.  Since $\mathcal{J}_1 \cup \mathcal{J}_2 = \nat$, it follows that
there is at least one possible choice of $i \in \{1,2\}$ for which
$\mathcal{L}$ has an infinitely large intersection with $\mathcal{J}_i$.  For
both possible values of $i$, the analyses above show that
$\lim_{j\to\infty,j\in\mathcal{J}_i} U_{k,j} = 0$, so it follows that
$\limsup_{j\to\infty} U_{k,j} = 0$, and, since $U_{k,j}$ is necessarily
nonnegative for all $j$, that $\lim_{j\to\infty} U_{k,j} = 0$.
\end{proof}

\section{Proof of Proposition~\ref{prop:fullAbstract}\ref{item:stuckInner}}
\label{app:full-generic-stuck-inner}
\begin{proof}
In the situation that the inner loop runs indefinitely at outer iteration $k$, 
Lemma~\ref{lem:genericInnerLoop}
asserts that $\lim_{j\to\infty} U_{k,j} = \lim_{j\to\infty} S_{k,j}
=\lim_{j\to\infty} A_{k,j} = 0$. From the definitions of the $U_{k,j}$ and
$S_{k,j}$, it immediately follows that $\lim_{j\to\infty} u^{k,j} = 0$ and
$\lim_{j\to\infty} s^{k,j} = 0$. From Definition~\ref{def:subSolver}, we also
have that
\begin{align}
  (\forall\,j \geq 0) && &&
  (s^{k,j}, 0) &\in \partial F(x^{k,j}, u^{k,j}) + (0,c_k u^{k,j} - p^k) \nonumber \\
  && \Leftrightarrow &&
  (s^{k,j}, p^k - c_k u^{k,j}) &\in \partial F(x^{k,j}, u^{k,j}). &&&& \label{cspsubgrad2}
\end{align}
Now, a vector $(\bar s,\bar p)\in\realn\times\realm$ is
a subgradient of $F$ at $(\bar x,\bar u)\in\realn\times\realm$ whenever 
\begin{equation*}
    (\forall\, (x,u) \in \realn) 
         \qquad F(x,u) \geq F(\bar x, \bar u) 
            + \big\langle(\bar s,\bar p),(x,u)-(\bar x,\bar u)\big\rangle. %
\end{equation*}
Fix any $j \geq 0$. Applying the above subgradient inequality with $u=0$,
$\bar x = x^{k,j}$, $\bar u = u^{k,j}$, $\bar s = s^{k,j}$, and $\bar p = p^k
- c_k u^{k,j}$ yields from~\eqref{cspsubgrad2} that
\begin{align}
(\forall\,x\in\real^n) \quad
    F(x, 0) &\geq F(x^{k,j}, u^{k,j}) + \inner{(s^{k,j}, p^k - c_k u^{k,j})}{(x, 0) - (x^{k,j}, u^{k,j})} \nonumber \\
    &\geq F(x^{k,j}, u^{k,j}) + \inner{s^{k,j}}{x - x^{k,j}} - \inner{p^k - c_k u^{k,j}}{u^{k,j}} \nonumber \\
    &\geq F(x^{k,j}, u^{k,j}) + \inner{s^{k,j}}{x} - \inner{s^{k,j}}{x^{k,j}} - \inner{p^k}{u^{k,j}} + c_k \norm{u^{k,j}}^2. \label{asymp-ineq} 
\end{align}
Now consider the behavior of the last four terms on the right-hand side of
this inequality as $j\to\infty$:
\begin{itemize}[nosep]
\item Since $\lim_{j\to\infty} s^{k,j} = 0$, it follows that
        $\lim_{j\to\infty} \inner{s^{k,j}}{x} = 0$ for any choice of $x\in\real^n$.
\item By the definition of compatible subproblem process, 
        $\lim_{j\to\infty} \inner{x^{k,j}}{s^{k,j}}=0$
\item Since  $\lim_{j\to\infty} u^{k,j} = 0$, one has 
        $\lim_{j\to\infty} \inner{p^k}{u^{k,j}} = 0$.
\item For the same reason, one immediately has
        $\lim_{j\to\infty} c_k \smallnorm{u^{k,j}}^2=0$.
\end{itemize}
Thus, all of these terms converge to zero with $j$, and so taking
the $j\to\infty$ limit in~\eqref{asymp-ineq} leads to
\begin{equation*}
    (\forall\,x\in\real^n) \qquad F(x, 0) \geq \limsup_{j\to\infty} F(x^{k,j}, u^{k,j}).
\end{equation*}
Taking the infimum on the left over all $x\in\real^n$, one then obtains.
\begin{equation*}
\inf_{x\in\real^n} F(x,0) \geq \limsup_{j\to\infty} F(x^{k,j}, u^{k,j}).
\end{equation*}
This result, together with $\lim_{j\to\infty} u^{k,j} = 0$ as already
established, is the claimed asymptotic optimality primal property of
$\{x^{k,j}\}_{j=0}^\infty$.

By~\eqref{eq:subgradient-relations}, an equivalent statement 
to~\eqref{cspsubgrad2} is that $(x^{k,j}, u^{k,j}) \in
\partial D(s^{k,j}, p^k - c_k u^{k,j})$ for all $j$. Applying the definition of
concave subgradient mapping $\partial D$ in~\eqref{eq:subgradient-dual}
with $s'=0$ then produces
\begin{align*}
(\forall\,p\in\real^m) \;\;
    D(0, p^*) &\leq D(s^{k,j}, p^k - c_k u^{k,j}) - \inner{x^{k,j}}{0-s^{k,j}} - \inner{u^{k,j}}{p^* - (p^k - c_k u^{k,j})} \\
    &\leq D(s^{k,j}, p^k - c_k u^{k,j}) + \inner{x^{k,j}}{s^{k,j}} - \inner{u^{k,j}}{p^*} + \inner{u^{k,j}}{p^k} - c_k \norm{u^{k,j}}^2.
\end{align*}
By logic similar to that applied to~\eqref{asymp-ineq}, the last four terms on
the right side of this inequality converge to zero as $j\to\infty$, so
taking the limit as $j\to\infty$ yields
\begin{align*}
(\forall\,p\in\real^m) \quad
    D(0, p) &\leq \liminf_{j\to\infty} D(s^{k,j}, p^k - c_k u^{k,j}) \\
    & \leq \limsup_{j\to\infty} D(s^{k,j}, p^k - c_k u^{k,j}) \\
    & \leq D\Big(\lim_{j\to\infty} s^{k,j},\lim_{j\to\infty} \{p^k - c_k u^{k,j}\} \Big) &&[\text{by upper semicontinuity of~}D]
         \\
    & \leq D(0, p^k) &&[\text{since~} s^{k,j}\to 0, u^{k,j}\to 0].
\end{align*}
Thus, $p^k$ must be a dual optimal point.
\end{proof}

\section{Proof of Lemma~\ref{lemma:derivative-fk}}\label{app:derivative-fk-proof}
\begin{proof}
Fix any $k\geq 0$ and $z\in\real^m$. Fenchel's inequality states that, for
any $p\in\real^m$, it is always the case that $h_k(p) + h_k^*(-z) \geq
\inner{p}{-z}$, and that $p\in\partial h_k^*(-z)$ if and only if $h_k(p) +
h_k^*(-z) = \inner{p}{-z}$.  An equivalent statement is that $h_k^*(-z)
\geq -h_k(p) - \inner{p}{z}$ for all $p\in\real^m$ and $p\in\partial
h_k^*(-z)$ if and only if $h_k^*(-z) = -h_k(p) - \inner{p}{z}$. Therefore,
\begin{align}
        \partial h_k^*(-z) 
            & = \argmax_{p\in\real^m} 
                    \big\{ - h_k(p) -\inner{p}{z}  \big\} \nonumber \\
            & = \argmin_{p\in\real^m} 
                    \big\{\inner{p}{z} + h_k(p) \big\} \nonumber \\
            & = \argmin_{p\in\real^m} 
                    \left\{\inner{p}{z} + f^*(-M\transpose p) 
                            + \frac{1}{2c_k}\norm{p - p^k}^2 \right\}.
                            \label{eq:problem-hk}
\end{align}
The minimization in~\eqref{eq:problem-hk} has a unique solution since the
function being minimized is closed, proper, and strongly convex, so therefore
the set of subgradients $\partial h_k^*(-z)$ of $h_k^*$ at $-z$ is a
singleton.  When a convex function has exactly one subgradient at a point,
then the function is differentiable there and the given subgradient is its
gradient; see for example~\cite[Theorem 25.1]{rockafellar1970convex}
or~\cite[Proposition 4.2.2]{BerConv}. Therefore, $\nabla h_k^*(-z)$ is the
unique vector $t\in\real^m$ meeting the optimality conditions
for~\eqref{eq:problem-hk}, namely
\begin{equation} \label{eq:dual-gradient-1}
    0 \in z -M \partial f^*(-M\transpose t) + \frac{1}{c_k}(t-p^k).
\end{equation}
On the other hand, a necessary optimality condition for~\eqref{eq:x-min} is
\begin{align*}
&&    0 &\in  \partial f(x) + M\transpose p^k + c_kM\transpose(M\bar x - z) && \\
\Leftrightarrow && 
        -M\transpose\big(p^k + c_k(M\bar x - z) \big) &\in \partial f(x) \\
\Leftrightarrow &&
        x &\in \partial f^*\big(-M\transpose\big(p^k + c_k(M\bar x-z)\big)\big).
\end{align*}
Left-multiplying both sides of this inclusion by $-c_k M$ produces
\begin{equation*}
-c_k M \bar x 
    \in -c_k M \partial f^*\big(-M\transpose\big(p^k + c_k(M\bar x-z)\big)\big).
\end{equation*}
Adding $c_k z - p^k$ to both sides of this inclusion then leads to
\begin{equation*}
    -\big(p^k + c_k(M\bar x-z)\big)
    \in -p^k + c_k\Big(z - M\partial f^*\big(-M\transpose\big(p^k + c_k(M\bar x-z)\big)\big)\Big).
\end{equation*}
Letting $s \doteq p^k + c_k(M\bar x-z)$, the above inclusion is equivalent to
\begin{align*}
&&    -s &\in -p^k + c_k\big(z - M\partial f^*(-M\transpose s)\big) && \\
\Leftrightarrow &&
        0 &\in s - p^k + c_k\big(z - M\partial f^*(-M\transpose s)\big)  \\
\Leftrightarrow &&
        0 &\in z - M\partial f^*(-M\transpose s) + \frac{1}{c_k}(s - p^k).
\end{align*} 
Since this inclusion is identical to~\eqref{eq:dual-gradient-1}, whose
solution is unique, it must be the case $\nabla h_k^*(-z) = t = s = p^k +
c_k(M\bar x-z)$.  Applying the chain rule to $f_k = h_k^* \circ (-I)$ then yields
\begin{equation} \label{eq:gradient-fk}
    \nabla f_k(z) = - \nabla h_k^*(-z)=-s=-p^k - c_k(M\bar x-z),
\end{equation} 
as claimed.
\end{proof}

\section{Proof of Proposition~\ref{prop:alternatingCSP}}
\label{app:alternatingCSP-proof}
\begin{proof}
Fix any $j \geq 0$. From the necessary optimality conditions
for~\eqref{eq:admm-xkj-again}, one may derive that
\begin{equation*}
    0 \in \partial f(x^{k,j+1}) + M\transpose p^k + c_k M\transpose(Mx^{k,j+1} - z^{k,j}).
\end{equation*}
Adding $c_k M\transpose (z^{k,j} - z^{k,j+1})$ to both sides of this inclusion
yields
\begin{equation} \label{fb-x-subgrad}
    c_k M\transpose (z^{k,j} - z^{k,j+1}) \in \partial f(x^{k,j+1}) + M\transpose p^k + c_k M\transpose (Mx^{k,j+1} - z^{k,j+1}).
\end{equation}
Similarly, from the necessary optimality conditions for~\eqref{eq:admm-zkj}, we have
\begin{equation} \label{fb-z-subgrad}
0 \in \partial g(z^{k,j+1}) - p^k + c_k(z^{k, j+1} - M x^{k, j+1})
= g(z^{k,j+1}) -\big(p^k + c_k(M x^{k, j+1} - z^{k, j+1}) \big).
\end{equation}
Applying~\eqref{eq:conj-lag-sub} with $x=x^{k,j+1}$, $z=z^{k,j+1}$, and $p=p^k
+ c_k(Mx^{k, j+1}-z^{k, j+1})$ then produces
\begin{multline*}
    \partial L\Big(\big(x^{k,j+1},z^{k,j+1}\big), p^k + c_k(Mx^{k, j+1}-z^{k, j+1})\Big) \\
    = \Big\{\partial f(x^{k,j+1}) + M\transpose \big(p^k + c_k(Mx^{k, j+1}-z^{k, j+1})\big)\Big\} \\
    \qquad \qquad \qquad \times \Big\{ \partial g(z^{k,j+1}) - \big(p^k + c_k(Mx^{k, j+1}-z^{k, j+1})\big) \Big\} \\ 
    \times \big\{z^{k, j+1} - Mx^{k, j+1}\big\}.
\end{multline*}
Substituting~\eqref{fb-x-subgrad} as a member of the first set in this
Cartesian product and~\eqref{fb-z-subgrad} as a member of the second produces
\begin{equation*}
    \partial L\Big(\big(x^{k,j+1},z^{k,j+1}\big), p^k + c_k(Mx^{k, j+1}-z^{k, j+1})\Big) \\
    \ni \begin{pmatrix} c_k M\transpose (z^{k,j} - z^{k,j+1}) \\ 
                        0
                        \\ z^{k, j+1} - Mx^{k, j+1} 
        \end{pmatrix}.
\end{equation*}
Consequently, setting $u^{k,j+1}\doteq z^{k,j+1} - Mx^{k,j+1}$ and 
\begin{equation*}
    s^{k,j+1} \doteq 
    \begin{pmatrix}
        s_x^{k,j+1} \\
        0
    \end{pmatrix} \doteq
    \begin{pmatrix}
        c_k M\transpose (z^{k,j} - z^{k,j+1}) \\
        0
    \end{pmatrix}
\end{equation*}
as specified in~\eqref{subgrad-fb}, we obtain the equivalent inclusions
\begin{align*}
    && \Big((s_x^{k,j+1}, 0), u^{k,j} \Big) & \in \partial L\Big(\big(x^{k,j+1},z^{k,j+1}\big), p^k - c_k u^{k,j}\Big) \\
    \Leftrightarrow && \Big((s_x^{k,j+1}, 0), p^k - c_k u^{k,j} \Big) & \in \partial F\Big(\big(x^{k,j+1},z^{k,j+1}\big), u^{k,j}\Big) \\
    \Leftrightarrow && \Big((s_x^{k,j+1}, 0), 0 \Big) & \in \partial F\Big(\big(x^{k,j+1},z^{k,j+1}\big), u^{k,j}\Big) + \Big((0,0), c_k u^{k,j} - p^k\Big).
\end{align*}
Therefore, \eqref{CSPsubgrad} holds for an arbitrary choice of $j\geq 0$. To
establish a compatible subproblem process, we must also
verify~\eqref{CSPconverge} and~\eqref{CSPinner}.  

In view of the above analysis leading to
~\eqref{eq:admm-like-x-nukj1}-\eqref{eq:admm-like-z-nukj1}, the sequence
$\{z^{k,j}\}_{j=0}^\infty$ may be considered as being generated by the
proximal gradient method applied minimizing $f_k + g$, with its stepsize set
to $1/c_k$ and $c_k$ being a valid Lipschitz modulus of $\nabla f_k$. A
minimizer of this problem is assumed to exist, so the convergence of
$\{z^{k,j}\}_{j=0}^\infty$ to such a minimizer is guaranteed by standard
convergence results for the proximal gradient method, for
example~\cite[Theorem 10.24]{BeckBook}.  It then follows immediately from
their definitions that $s_x^{k,j}$ and $s^{k,j}$ both converge to zero as
$j\to\infty$, confirming~\eqref{CSPconverge}.

To confirm~\eqref{CSPinner}, observe that
\begin{align}
(\forall\,j\geq 0) \qquad 
\inner{(x^{k,j+1}, z^{k,j+1})}{s^{k,j+1}}
&= \inner{(x^{k,j+1}, z^{k,j+1})}{(s_x^{k,j+1},0)} \nonumber \\
&= \inner{x^{k,j+1}}{s_x^{k,j+1}} \nonumber \\
&= c_k\inner{x^{k,j+1}}{M\transpose (z^{k,j} - z^{k,j+1})} &&&& \nonumber \\
&= c_k\inner{Mx^{k,j+1}}{z^{k,j} - z^{k,j+1}}. \label{fb-inner-2}
\end{align}
We will show that the first vector in this inner product converges with $j$:
for any $j \geq 0$, one may rearrange the equation $\nabla f_k(z^{k,j}) =
-\big(p^k + c_k(Mx^{k,j+1}-z^{k,j})\big)$ from Lemma~\ref{lemma:derivative-fk}
into
\begin{equation*}
Mx^{k,j+1} = z^{k,j} - \frac{1}{c_k}\big(p^k + \nabla f_k(z^{k,j})\big).
\end{equation*}
By the continuity of $\nabla f_k$, as established just
below~\eqref{eq:defbal-prob}, convergence of $z^{k,j}$ as $j\to\infty$ thus
implies convergence of $Mx^{k,j}$, and we will call this limit
$(Mx)^{k,\infty}$. By the convergence of the $z^{k,j}$, the limit
of~\eqref{fb-inner-2} as $j\to\infty$ must be $c_k\inner{(Mx)^{k,\infty}}{0} =
0$, verifying~\eqref{CSPinner}.
\end{proof}

\section{Proof of Theorem~\ref{thm:ALM-AR-FISTA-CD-convergence}}
\label{app:ALM-AR-FISTA-CD-convergence-proof}
\begin{proof}
Considering the analysis in Section~\ref{sec:fista-cd}, 
the ALM-AR-FISTA-CD algorithm follows the generic 
Algorithm~\ref{alg:gen-relax} framework, using DEFBAL-CD as the 
subproblem solver. 
We will now show that the FISTA-CD subproblem solver is a compatible
subproblem process as presented in Definition~\ref{def:subSolver}. Fix any $k
\geq 0$ and define, for $\beta_j=(t_{k,j}-1)/{t_{k,j+1}}$ for all $j\geq 1$,
where the sequence $\{t_{j,k}\}$ is as specified in
steps~\ref{step:innerLoopInit} and~\ref{step:nextt}. From the recursion in
step~\ref{step:nextt}, one has 
\begin{equation}
(\forall\, j\geq 1) \qquad
    \beta_j=\frac{t_j-1}{t_{j+1}}
       =\frac{\frac{j+a-1}{a}-1}{\frac{j+1+a-1}{a}}=\frac{j-1}{j+a}. \label{defbetaj}
\end{equation}
From step~\ref{step:calcy}, we can then write
\begin{equation*}
(\forall\, j\geq 2) \qquad
y^{k,j}=z^{k,j}+\beta_{j-1}(z^{k,j}-z^{k,j-1}).
\end{equation*}
Recalling~\eqref{eq:subgradient-lagrange-kj0} and~\eqref{simplifiedSubgrad},
the $x$-subgradients $s_x^{k,j+1}$ resulting after steps~\ref{step:minx}
and~\ref{step:minz} may be rewritten as
\begin{align}
(\forall\, j\geq 2)&&
    s_x^{k,j+1} &= c_k M\transpose (y^{k,j}-z^{k,j+1}) 
        \nonumber \\
    &&&=c_k M\transpose \big( z^{k,j}+\beta_{j-1}(z^{k,j}-z^{k,j-1}) - z^{k,j+1} \big) 
        \nonumber \\
    &&&=c_k M\transpose \big( \beta_{j-1}(z^{k,j}-z^{k,j-1}) - (z^{k,j+1}-z^{k,j})\big).
        \label{expandedsxkjp1} 
    &&&&
\end{align}
By Theorem~\ref{thm:fista-cd-convergence}, we know $\{z^{k,j}\}$ converges to
some $z^{k,\infty}$, so both $z^{k,j+1}-z^{k,j}$ and $z^{k,j}-z^{k,j-1}$
converge to $0$. Since $a$ is positive, we can easily verify
from~\eqref{defbetaj} that $0 < \beta_j<1$ for all $j \geq 2$, so $y^{k,j}$
must converge to $z^{k,\infty}$, the same limit as for $z^{k,j}$. From the
boundedness of $\beta_j$ and $z^{k,j+1}-z^{k,j}$ and $z^{k,j}-z^{k,j-1}$
both converging to $0$, it follows that 
\begin{equation} \label{sxkjInsideLimit}
\lim_{j\to\infty} \big\{\beta_{j-1}(z^{k,j}-z^{k,j-1}) - (z^{k,j+1}-z^{k,j})\big\} = 0,
\end{equation}
and hence from~\eqref{expandedsxkjp1} that $\lim_{j \to \infty} s_x^{k,j+1} = 0$.
Again referring to~\eqref{eq:subgradient-lagrange-kj0}, the $z$-subgradient
$s_z^{k,j+1}$ resulting from step~\ref{step:minz} is $0$ for all $k$, so we
obtain 
\begin{equation*}
\lim_{j\to\infty} s^{k,j} 
= \lim_{j\to\infty} (s_x^{k,j},s_z^{k,j})
= \lim_{j\to\infty} (s_x^{k,j},0)
= (0,0).
\end{equation*}
Therefore, condition~\eqref{CSPconverge} from Definition~\ref{def:subSolver}
holds.  We next turn to condition~\eqref{CSPinner}. We have already shown that
$f_k$ is differentiable and with a $c_k$-Lipschitz gradient. %
Since $y^{k,j}$ converges to $z^{k,\infty}$,the continuity of $\nabla f_k$ implies
that $\nabla f_k(y^{k,j})$ converges to $\nabla f_k(z^{k,\infty})$. From 
Lemma~\ref{lemma:derivative-fk}, we know 
\begin{equation*}
(\forall\,j\geq 2) \qquad
    \nabla f_k(y^{k,j}) = -p^k - c_k (Mx^{k,j+1} - y^{k,j}).  
\end{equation*}
As $p^k$ and $c_k$ are fixed and $y^{k,j}$ converges, the convergence of
$\nabla f_k(y^{k,j})$ implies that $Mx^{k,j+1}$ is converging to some point
$(Mx)^{k,\infty}$. Now, from~\eqref{expandedsxkjp1} we may write 
\begin{align*}
(\forall\,k \geq 2) \quad
    \inner{x^{k,j}}{s_x^{k,j}} 
        &= \inner{x^{k,j}}
                 {c_k M\transpose\big( \beta_{j-1}
                      (z^{k,j}-z^{k,j-1}) - (z^{k,j+1}-z^{k,j})}  \\
    &= c_k \inner{Mx^{k,j}}{\beta_{j-1}(z^{k,j}-z^{k,j-1}) - (z^{k,j+1}-z^{k,j})}.
\end{align*}
Taking the limit $j\to\infty$ and using~\eqref{sxkjInsideLimit} then yields
\begin{equation*}
    \lim_{j\to\infty} \inner{x^{k,j}}{s_x^{k,j}}
    = \big\langle(Mx)^{k,\infty},0\big\rangle = 0.
\end{equation*}
Then, using that $s_z^{k,j} = 0$ for all $j$, one immediately has
\begin{equation*}
\lim_{j\to\infty} \big\langle(x^{k,j}, z^{k,j}),(s_x^{k,j},0)\big\rangle
= \lim_{j\to\infty} \big\{ \inner{x^{k,j}}{s_x^{k,j}} + \inner{z^{k,j}}{0} \big\}
= \lim_{j\to\infty} \inner{x^{k,j}}{s_x^{k,j}} = 0,
\end{equation*}
so condition~\eqref{CSPinner} does indeed hold.

It remains to establish condition~\eqref{CSPsubgrad}. 
From~\eqref{eq:conj-lag-sub} and the optimality conditions of~\eqref{eq:admm-xkj-cd}
and~\eqref{eq:admm-zkj-cd}, we can write, for any choice of $j \geq 1$,
\begin{align*}
    && \big((s^{k,j+1}, 0), z^{k,j+1} - Mx^{k,j+1}\big) 
       &\in \partial L\big((x^{k,j+1}, z^{k,j+1}), p^k+c_k(Mx^{k,j+1}-z^{k,j+1})\big) \\
    \Leftrightarrow && \big((s_x^{k,j+1}, 0), u^{k,j+1})\big) 
       &\in \partial L\big((x^{k,j+1}, z^{k,j+1}), p^k-c_k u^{k,j+1}\big) \\
    \Leftrightarrow && \big((s_x^{k,j+1}, 0), p^k-c_k u^{k,j+1}\big) 
       &\in \partial F\big((x^{k,j+1}, z^{k,j+1}), u^{k,j+1}\big) \\
    \Leftrightarrow && \big((s_x^{k,j+1}, 0), 0\big) 
       &\in \partial F\big((x^{k,j+1}, z^{k,j+1}), u^{k,j+1}\big) 
                     + \big((0, 0), c_k u^{k,j+1} - p^k\big).
\end{align*}
Keeping in mind that $s^{k,j} = (s_x^{k,j},0)$ for all $k$, the final
condition above verifies~\eqref{CSPsubgrad}. As a result, the FISTA-CD
procedure is a compatible subproblem process.

All the claimed conclusions then follow immediately from
Proposition~\ref{prop:convergeUSA} and the specific form of the dual problem
derived in Section~\ref{sec:sum-of-two-convex}.
\end{proof}

\end{document}